\documentclass[11pt]{amsart}
\usepackage{amsmath,amsfonts,amssymb,upgreek, amscd, mathtools}
\usepackage[breaklinks]{hyperref}

\theoremstyle{thrm}
\usepackage{tikz}
\usepackage{enumerate}
\usetikzlibrary{calc,decorations.markings}
\usepackage{tikz-cd}
\usepackage{graphicx}

\theoremstyle{plain}
\newtheorem{no}{Notation}[section]

\newtheorem{thm}{Theorem}[section]
\newtheorem{lemma}[thm]{Lemma}
\newtheorem{prop}[thm]{Proposition}

\newtheorem{cor}[thm]{Corollary}

\newtheorem{defn}[thm]{Definition}

\newtheorem*{ack}{Acknowledgement}

\theoremstyle{definition}
\newtheorem{remark}[equation]{Remark}

\newcommand{\B}{\operatorname{B} }
\newcommand{\C}{\operatorname{C} }

\newcommand{\Ho}{\operatorname{H} }

\newcommand{\Z}{\operatorname{Z} }

\newcommand{\im}{\operatorname{im} }

\newcommand{\Fix}{\operatorname{Fix}}

\newcommand{\Id}{\operatorname{id}}

\newcommand{\I}{\operatorname{I} }

\newcommand{\Aut}{\operatorname{Aut} }

\newcommand{\Tra}{\operatorname{Tra} }
\newcommand{\Res}{\operatorname{Res} }
\newcommand{\Inf}{\operatorname{Inf} }

\newcommand{\Ext}{\operatorname{Ext} }

\newcommand{\CExt}{\operatorname{CExt} }

\newcommand{\Hom}{\operatorname{Hom} }

\newcommand{\Ker}{\operatorname{ker} }

\newcommand{\Ann}{\operatorname{Ann} }

\newcommand{\IM}{\operatorname{im} }

\numberwithin{equation}{section}

\begin{document}

\title{Schur multiplier and Schur covers of relative Rota-Baxter groups}

\author{Pragya Belwal}
\address{Department of Mathematical Sciences, Indian Institute of Science Education and Research (IISER) Mohali, Sector 81, SAS Nagar, P O Manauli, Punjab 140306, India} 
\email{pragyabelwal.math@gmail.com}

\author{Nishant Rathee}
\address{Department of Mathematical Sciences, Indian Institute of Science Education and Research (IISER) Mohali, Sector 81, SAS Nagar, P O Manauli, Punjab 140306, India} 
\email{nishantrathee@iisermohali.ac.in}

\author{Mahender Singh}
\address{Department of Mathematical Sciences, Indian Institute of Science Education and Research (IISER) Mohali, Sector 81, SAS Nagar, P O Manauli, Punjab 140306, India} 
\email{mahender@iisermohali.ac.in} 

\subjclass[2010]{17B38, 16T25, 81R50}
\keywords{Cohomology; extension; isoclinism; relative Rota-Baxter group; Schur multiplier; Schur cover; skew left brace; Yang-Baxter equation}

\begin{abstract}
Relative Rota-Baxter groups are generalizations of Rota-Baxter groups and share a close connection with skew left braces. These structures are well-known for offering bijective non-degenerate set-theoretical solutions to the Yang-Baxter equation. This paper builds upon the recently introduced extension theory and low-dimensional cohomology of relative Rota-Baxter groups. We prove an analogue of the Hochschild-Serre exact sequence for central extensions of relative Rota-Baxter groups. We introduce the Schur multiplier  $M_{RRB}(\mathcal{A})$ of a  relative Rota-Baxter group $\mathcal{A} =(A,B,\beta,T)$, and prove that the exponent of $M_{RRB}(\mathcal{A})$ divides $|A||B|$ when $\mathcal{A}$ is finite. We define weak isoclinism of relative Rota-Baxter groups, introduce their Schur covers, and prove that any two Schur covers of a finite bijective relative Rota-Baxter group are weakly isoclinic. The results align with recent results of Letourmy and Vendramin for skew left braces.
\end{abstract}	

\maketitle

\section{Introduction}

A set-theoretical solution to the Yang-Baxter equation is a pair $(X, r)$, consisting of a set $X$ and a map $r: X\times X \longrightarrow X\times X$ expressed as $r(x,y) = (\sigma_x(y), \tau_y(x))$, such that the identity
$$(r\times \Id)(\Id \times r)(r\times \Id)=(\Id \times r)(r\times \Id)( \Id \times r)$$
holds. A solution is categorized as non-degenerate when the maps $ \sigma_x$ and $\tau_x$ are bijective for all $x \in X$.  In \cite{MR1183474}, Drinfel'd proposed to classify all set-theoretical solutions to the Yang-Baxter equation in order to understand its solutions in full generality.
\par 

Rota-Baxter operators of weight 1 on Lie groups were introduced by Guo, Lang and Sheng in \cite{LHY2021}, with a focus on smooth Rota-Baxter operators. The purpose was to enable differentiation, yielding Rota-Baxter operators of weight 1 on the corresponding Lie algebras. Subsequent work on Rota-Baxter operators in (abstract) groups has been conducted by Bardakov and Gubarev in \cite{VV2022, VV2023}. They proved that every Rota-Baxter operator on a group gives rise to a skew left brace structure on that group, providing a set-theoretical solution to the Yang-Baxter equation. Further, extensions and automorphisms of Rota-Baxter groups have been explored in \cite{AN2}. Skew left braces, central to the classification of solutions to the Yang-Baxter equation, were introduced by Guarnieri and Vendramin \cite{GV17}. It is established therein that every skew left brace leads to a bijective non-degenerate solution to the Yang-Baxter equation.
\par

In a recent work \cite{BRS2023}, the authors developed an extension theory and introduced low-dimensional cohomology of relative Rota-Baxter groups. This theory differs from those known in the context of Rota-Baxter operators on Lie groups. The present paper further develops this cohomology theory in the spirit of Schur \cite{schur1904, schur1907}, and is also motivated by the recent work of Letourmy and Vendramin on skew left braces  \cite{TV23}. By utilizing the idea of isoclinism for skew left braces, as outlined in \cite{TV22}, they establish in \cite{TV23} that distinct Schur covers corresponding to a finite skew left brace are isoclinic. In this paper, we extend this construction to the domain of relative Rota-Baxter groups through the recent classification of abelian extensions of relative Rota-Baxter groups via the second cohomology group, as formulated in \cite{BRS2023}.
\par

The paper is organised as follows. We recall the necessary background on relative Rota-Baxter groups in Section \ref{section prelim}. Further, the recently introduced extension theory and cohomology for relative Rota-Baxter groups and their connections with analogous objects for groups and skew left braces is recalled in Section \ref{sec extensions and cohomology RRB}. In Section \ref{Hochschild-Serre sequence}, we  prove an analogue of the Hochschild-Serre exact sequence for central extensions of relative Rota-Baxter groups (Theorem \ref{HSS}).  In Section \ref{sec Schur multiplier and cover}, we introduce the Schur multiplier  $M_{RRB}(\mathcal{A})$ of a  relative Rota-Baxter group $\mathcal{A} =(A,B,\beta,T)$ and prove that the exponent of $M_{RRB}(\mathcal{A})$ divides $|A||B|$ when $\mathcal{A}$ is finite (Theorem \ref{VI1}), which generalises a classical result for groups. This result is then used to prove that every finite relative Rota-Baxter group admits a central extension whose corresponding transgression map is an isomorphism (Theorem \ref{existence of Transgression Isomorphism}). In Section \ref{isoclinic Schur covers}, we introduce the notion of weak isoclinism of relative Rota-Baxter groups and define their Schur covers. We prove that all central extensions of a finite relative Rota-Baxter group whose corresponding transgression maps are isomorphisms,  are weakly isoclinic (Theorem \ref{thm:almost isoclinism}). As a consequence,  we prove that every finite bijective relative Rota Baxter group admits at least one Schur cover, and that any two Schur covers are weakly isoclinic. 
\medskip

\section{Preliminaries on relative Rota-Baxter groups}\label{section prelim}

In this section, we recall some basic notions about relative Rota-Baxter groups that we shall need, and refer the readers to \cite{JYC22, NM1} for more details. We follow the  terminology of \cite{BRS2023, NM1}, which is a bit different from other works in the literature. 

\begin{defn}
	A relative Rota-Baxter group is a quadruple $(H, G, \phi, R)$, where $H$ and $G$ are groups, $\phi: G \rightarrow \Aut(H)$ a group homomorphism (where $\phi(g)$ is denoted by $\phi_g$) and $R: H \rightarrow G$ is a map satisfying the condition $$R(h_1) R(h_2)=R(h_1 \phi_{R(h_1)}(h_2))$$ for all $h_1, h_2 \in H$. 
	\par
	\noindent The map $R$ is referred as the relative Rota-Baxter operator on $H$.
\end{defn}

We say that the relative Rota-Baxter group  $(H, G, \phi, R)$ is {\it trivial} if $\phi:G \to \Aut(H)$ is the trivial homomorphism. Further, it is said to be {\it bijective} if the Rota-Baxter operator $R$ is a bijection. We refer the reader to \cite{BRS2023, JYC22} for examples.

\begin{remark}
	Note that if $(H, G, \phi, R)$ is  a trivial relative Rota-Baxter group, then $R:H \rightarrow G$ is a group homomorphism. Different homomorphisms give different trivial relative Rota-Baxter groups. Moreover, if $R$ is an isomorphism of groups, then $(H, G, \phi, R)$ is  a trivial relative Rota-Baxter group. 
\end{remark}

Let  $(H, G, \phi, R)$ be a relative Rota-Baxter group, and let $K \leq H$ and $L \leq G$ be subgroups.
\begin{enumerate}
	\item If  $\phi_\ell(K) \subseteq K$ for all $\ell \in L$, then we denote the restriction of $\phi$ by $\phi|: L \to \Aut(K)$. 
	\item If $R(K) \subseteq L$, then we denote the restriction of $R$ by $R|: K \to L$.
\end{enumerate}

\begin{defn}
	Let $(H,G,\phi,R)$ be a relative Rota-Baxter group, and $K\leq H$ and $L\leq G$ be subgroups. Suppose that  $\phi_\ell(K) \subseteq K$ for all $\ell \in L$ and $R(K) \subseteq L$. Then $(K,L,\phi |,R |)$ is a relative Rota-Baxter group, which we refer as a relative Rota-Baxter subgroup of $(H,G,\phi,R)$ and write $(K,L,\phi |,R |)\leq(H,G,\phi,R)$.
\end{defn}

\begin{defn}\label{defn ideal rbb-datum}
	Let $(H, G, \phi, R)$ be a relative Rota-Baxter group and  $(K, L,  \phi|, R|) \leq (H, G, \phi, R)$ its relative Rota-Baxter subgroup. We say that $(K, L,  \phi|, R|)$ is an ideal of $(H, G, \phi, R)$ if 
	\begin{align}
		& K \trianglelefteq H \quad \mbox{and} \quad L \trianglelefteq G, \label{I0}\\
		& \phi_g(K) \subseteq K  \mbox{ for all } g \in G, \label{I1} \\
		& \phi_\ell(h) h^{-1} \in K \mbox{ for all } h \in H \mbox{ and }  \ell \in L. \label{I2}
	\end{align}
	We write $(K, L, \phi|, R|) \trianglelefteq (H, G, \phi, R)$ to denote an ideal of a relative Rota-Baxter group. 
\end{defn}

The preceding definitions lead to the following result \cite[Theorem 5.3]{NM1}.

\begin{thm}\label{subs}
	Let $(H, G, \phi, R)$ be a relative Rota-Baxter group and $(K, L,  \phi|, R|)$ an ideal of $(H, G, \phi, R)$. Then there are maps $\overline{\phi}: G/L \to \Aut(H/K)$ and  $\overline{R}: H/K \to G/L$ defined by
	$$ \overline{\phi}_{\overline{g}}(\overline{h})=\overline{\phi_{g}(h)} \quad \textrm{and} \quad	\overline{R}(\overline{h})=\overline{R(h)}$$
	for $\overline{g} \in G/L$ and $\overline{h} \in H/K$, such that  $(H/K, G/L, \overline{\phi}, \overline{R})$ is a relative Rota-Baxter group.
\end{thm}
\begin{no}
	We write $(H, G, \phi, R)/(K, L, \phi|, R|)$ to denote the quotient relative Rota-Baxter group $(H/K, G/L, \overline{\phi}, \overline{R})$.
\end{no}

\begin{defn}
	Let $(H, G, \phi, R)$ and $(K, L, \varphi, S)$ be two relative Rota-Baxter groups.
	\begin{enumerate}
		\item A homomorphism $(\psi, \eta): (H, G, \phi, R) \to (K, L, \varphi, S)$ of relative Rota-Baxter groups is a pair $(\psi, \eta)$, where $\psi: H \rightarrow K$ and $\eta: G \rightarrow L$ are group homomorphisms such that
		\begin{equation}\label{rbb datum morphism}
			\eta \; R = S \; \psi \quad \textrm{and} \quad \psi \; \phi_g =  \varphi_{\eta(g)}\psi
		\end{equation}
		for all $g \in G$.
		\item The kernel of a homomorphism $(\psi, \eta): (H, G, \phi, R) \to (K, L, \varphi, S)$ of relative Rota-Baxter groups is the quadruple $$(\Ker(\psi), \Ker(\eta), \phi|, R|),$$ where $\Ker(\psi)$ and $\Ker(\eta)$ denote the kernels of the group homomorphisms $\psi$ and $\eta$, respectively. The conditions in \eqref{rbb datum morphism} imply that the kernel is itself a relative Rota-Baxter group. In fact, the kernel turns out to be an ideal of $(H, G, \phi, R)$.
		
		\item The image of a homomorphism $(\psi, \eta):  (H, G, \phi, R) \to (K, L, \varphi, S)$ of relative Rota-Baxter groups is the quadruple 
		$$(\IM(\psi), \IM(\eta), \varphi|, S| ),$$ where $\IM(\psi)$ and $\IM(\eta)$ denote the images of the group homomorphisms $\psi$ and $\eta$, respectively. The image is itself a relative Rota-Baxter group.
		
		\item A homomorphism $(\psi, \eta)$ of relative Rota-Baxter groups is called an isomorphism if both $\psi$ and $\eta$ are group isomorphisms. Similarly, we say that $(\psi, \eta)$ is an embedding of a relative Rota-Baxter group if both $\psi$ and $\eta$ are embeddings of groups.
	\end{enumerate}
\end{defn}

\begin{defn}
	A Rota-Baxter group is a group $G$ together with a map $R: G \rightarrow G$ such that
	$$ R(x)  R(y)= R(x  R(x)  y  R(x)^{-1}) $$
	for all $x, y \in G$. The map $R$ is referred as the Rota-Baxter operator on $G$.
\end{defn}
Let $\phi : G \rightarrow \Aut(G)$ be the adjoint action, that is, $\phi_g(x)=gxg^{-1}$ for $g, h \in G$. Then the relative Rota-Baxter group $(G, G, \phi, R)$ is simply a Rota-Baxter group. For convenience, we shall often write a group $(G, \cdot)$ as $G^{(\cdot)}$.

\begin{prop}\label{R homo H to G} \cite[Proposition 3.5]{JYC22} 
	Let $(H, G, \phi, R)$ be a relative Rota-Baxter group. Then the operation 
	\begin{align}
		h_1 \circ_R h_2 = h_1 \phi_{R(h_1)}(h_2)
	\end{align}
	defines a group operation on $H$.  Moreover, the map $R: H^{(\circ_R)} \rightarrow G$ is a group homomorphism. The group $H^{(\circ_R)}$ is called the descendent group of $R$.
\end{prop}

It follows that if  $(H, G, \phi, R)$ is a relative Rota-Baxter group, then the image $R(H)$ of $H$ under $R$ is  a subgroup of $G$.

\begin{defn}
	The center of a relative Rota-Baxter group $(H, G, \phi, R)$ is defined as
	$$\Z(H, G, \phi, R)= \big( \Z^{\phi}_R(H), \Ker(\phi), \phi|, R| \big),$$
	where $\Z^{\phi}_{R}(H)=\Z(H)  \cap \Ker(\phi \,R) \cap \Fix(\phi)$, $\Fix(\phi)= \{ x \in H  \mid \phi_g(x)=x \mbox{ for all } g \in G \}$ and  $R: H^{(\circ_R)} \rightarrow G$ is viewed as a group homomorphism.
\end{defn}

Let us recall the definition of a skew left brace.

\begin{defn}
	A skew left brace is a  triple $(H,\cdot ,\circ)$, where $(H,\cdot)$ and $(H, \circ)$ are groups  such that
	$$a \circ (b \cdot c)=(a\circ b) \cdot a^{-1} \cdot (a \circ c)$$
	holds for all $a,b,c \in H$, where $a^{-1}$ denotes the inverse of $a$ in $(H, \cdot)$. The groups $(H,\cdot)$ and $(H, \circ)$ are called the additive and the multiplicative groups of the skew left brace $(H, \cdot, \circ)$.
\end{defn}

A skew left brace $(H, \cdot, \circ)$  is said to be  trivial  if $a \cdot b= a \circ b$ for all $a, b \in H$.

\begin{prop}\label{rrb2sb}\cite[Proposition 3.5]{NM1} 
	Let $(H, G, \phi, R)$ be a relative Rota-Baxter group. If $\cdot$ denotes the  group operation of $H$, then the triple $(H, \cdot, \circ_R)$ is a skew left brace. 
\end{prop}

If $(H, G, \phi, R)$ is a relative Rota-Baxter group, then $(H, \cdot, \circ_R)$ is referred as the skew left brace induced by $R$ and will be denoted by $H_R$ for brevity. The following indispensable result is immediate \cite[Proposition 4.3]{NM1}.

\begin{prop}\label{rrb to slb homo}
A homomorphism of relative Rota-Baxter groups induces a homomorphism of induced skew left braces. 
\end{prop}

In fact, the association $(H, \cdot, \circ_R) \mapsto H_R$ is a functor from the category of relative Rota-Baxter groups to that of skew left braces. 
\medskip

\section{Central Extensions and Second cohomology} \label{sec extensions and cohomology RRB}
In this section, we revisit extensions and second cohomology of algebraic structures related to Rota-Baxter groups and explore their connections.

\subsection{Central extensions and second cohomology of groups } We recall the particular case of the connection between group cohomology and abelian extensions. For additional details, see \cite{MR0672956} and  \cite[Chapter 2]{PSY18}.

\begin{defn}
	A central extension $\mathcal{E}$ of a group $A$ by an abelian group $K$ is a short exact sequence groups
	$$ \mathcal{E}: \quad {\bf 1} \longrightarrow K \stackrel{i}{\longrightarrow}  H \stackrel{\pi}{\longrightarrow} A \longrightarrow {\bf 1}$$
	such that $K \subseteq \Z(H)$.
\end{defn}

\medskip

Let $\Z^2_{Gp}(A, K)$ be the group of 2-cocycles, which consists of maps $\tau:A \times A \to K$ that vanish on degenerate tuples, that is, on tuples $(a,b)$ in which either $a=1$  or  $b=1$
and satisfy
$$\tau(a_2, a_3) \tau(a_1 a_2, a_3)^{-1} \tau(a_1, a_2 a_3) \tau(a_1, a_2)^{-1}=1$$
for $a_1, a_2, a_3 \in A.$ It is well-known and easy to see that each $\tau \in \Z^2_{Gp}(A, K)$ induces a group structure $A \times_{\tau_1} K$ on the set $A \times K$  defined by the binary operation
$$(a_1, k_1) (a_2, k_2)=(a_1 a_2, k_1k_2\tau(a_1, a_2))$$
for $a_1, a_2 \in A$ and $k_1, k_2 \in K.$
\par 
Let $\B^2_{Gp}(A, K)$ be the group of 2-coboundaries, which is the subgroup of $\Z^2_{Gp}(A, K)$ consisting of maps $\tau$ for which there exists a map $\theta: A \rightarrow K$ such that
$$\tau(a_1, a_2)=\partial^1(\theta):=\theta(a_2) \theta(a_1 a_2)^{-1} \theta(a_1)$$
for all $a_1, a_2 \in A$.  Then the second cohomology group of $A$ with coefficients in $K$ is defined by 
$$\Ho^2_{Gp}(A,K):=\Z^2_{Gp}(A, K)/ \B^2_{Gp}(A, K).$$

Let $\Ext_{Gp}(A,K)$ denotes the set of equivalence classes of all central extensions of  $A$ by $K$. Then the following result is well-known.
\begin{thm}
	The map $\Psi_1: \Ho^2_{Gp}(A,K) \longrightarrow \Ext_{Gp}(A,K)$ defined by $\Psi_1 \big([\tau] \big)=[\mathcal{E}(\tau)]$ is a bijection, where 
	$$ \mathcal{E}(\tau): \quad {\bf 1} \longrightarrow K \stackrel{i}{\longrightarrow}  A \times_{\tau} K  \stackrel{\pi}{\longrightarrow} A \longrightarrow {\bf 1},$$
and the maps $i$ and $ \pi$ represent the natural injection and the natural projection, respectively.
\end{thm}

\subsection{Annhilator extensions and second cohomology of skew left braces} 
In the category of skew left braces, the counterparts of central extensions are known as annihilator extensions (see \cite{CMP, NMY1}).

Let $(H, \cdot, \circ)$ be a skew left brace, $\Z(H^{(\cdot)})$ denote the centre of the group $H^{(\cdot)}$ and $\Fix(\lambda) = \{x \in H \mid \lambda_a(x) = x \quad \textrm{for all} \quad  a\in H \}$.  Then the annihilator $\Ann(H)$ of $(H,\cdot ,\circ)$ is defined as $$\Ann(H)=\Ker(\lambda) \cap \Z(H^{(\cdot)}) \cap \Fix(\lambda)=\{a \in H \mid b \circ a = a \circ b = b \cdot a = a \cdot b \quad \textrm{for all} \quad b \in H\}.$$ Clearly, $\Ann(H)$ is an ideal of  $(H, \cdot, \circ)$. We now introduce the definition of an annihilator extension for skew left braces.

\begin{defn}
	An annhilator extension $\mathcal{E}$ of a skew left brace $(M, \cdot, \circ)$ by an abelian group $I$ viewed as a trivial  brace is  a short exact sequence skew left braces
	$$ \mathcal{E}: \quad {\bf 1} \longrightarrow I \stackrel{i}{\longrightarrow}  E \stackrel{\pi}{\longrightarrow} M \longrightarrow {\bf 1}$$
	such that $I \subseteq \Ann(E)$.
\end{defn}

Let $\Z^2_{SLB}(M,I)$ be the subgroup of $\Z^2_{Gp}(M^{(\cdot)}, I) \bigoplus \Z^2_{Gp}(M^{(\circ)}, I)$ consisting of pairs $(\tau, \tilde{\tau})$ which satisfy the compatibility condition
$$\tau(m_2, m_3) \tilde{\tau}(m_1, m_2 \cdot m_3) \tau(m^{-1}_1, m_1)= \tilde{\tau}(m_1, m_3) \tau(m_1 \circ m_2, m^{-1}_1) \tau((m_1 \circ m_2)\cdot m^{-1}_1, m_1 \circ m_3 ) \tilde{\tau}(m_1, m_2)$$
for all $m_1, m_2,m_3 \in M.$ 

It follows from \cite[Theorem 3.4]{NMY1} that each $(\tau, \tilde{\tau}) \in \Z^2_{SLB}(M,I)$ defines a skew left brace structure $(M \times_{(\tau, \tilde{\tau})} I, \cdot, \circ)$ on the set $M \times I$  given by 
\begin{align*}
	(m_1, y_1) \cdot (m_2, y_2)=&\; (m_1 \cdot m_2, \, y_1y_2 \tau(m_1, m_2)),\\
	(m_1, y_1) \circ (m_2, y_2)=& \;(m_1 \circ m_2, \, y_1y_2\tilde{\tau}(m_1, m_2))
\end{align*} 
for all $m_1, m_2 \in M$ and $y_1, y_2 \in M.$
\par 
Let $\B^2_{SLB}(M, I)$ be the group of 2-coboundaries, which is a subgroup of $\Z^2_{SLB}(M, I)$ consisting of  $(\tau, \tilde{\tau})$ for which there exists a map $\theta: M \rightarrow I$ such that
\begin{align*}
	\tau(m_1, m_2) = \partial^1(\theta) =&\; \theta(m_2) \theta(m_1 \cdot m_2)^{-1} \theta(m_1),\\
	\tilde{\tau}(m_1, m_2)) = \partial^1(\theta) =& \;\theta(m_2) \theta(m_1 \circ m_2)^{-1} \theta(m_1)
\end{align*}
for all $m_1, m_2 \in M$.  Then the second cohomology group of the skew left brace $M$ with coefficients in $I$ is defined by 
$$\Ho^2_{SLB}(M,I):=\Z^2_{SLB}(M, I)/ \B^2_{SLB}(M, I).$$
\par
Let $\Ext_{SLB}(M,I)$ denote the set of equivalence classes of all annihilator extensions of the skew left brace $M$ by the abelian group $I$. Then the following result is known due to \cite[Theorem 3.6]{NMY1}.

\begin{thm}
	The map $\Psi_2: \Ho^2_{SLB}(M,I) \longrightarrow \Ext_{SLB}(M,I)$ given by 
	$$ \Psi_2 \big([(\tau, \tilde{\tau})] \big) =[\mathcal{E}(\tau, \tilde{\tau})],$$ 
	where 
	$$ \mathcal{E}(\tau, \tilde{\tau}): \quad {\bf 1} \longrightarrow I \stackrel{i}{\longrightarrow}  M \times_{(\tau, \tilde{\tau})} I  \stackrel{\pi}{\longrightarrow} M \longrightarrow {\bf 1},$$
is a bijection and the maps $i$ and $ \pi$ represent the natural injection and the natural projection, respectively.
\end{thm}

\subsection{Central extensions and second cohomology of relative Rota-Baxter groups}

In this subsection, we recall central extensions of relative Rota-Baxter groups and relevant basic results  from \cite{BRS2023}.

\begin{defn}
Let  $\mathcal{A}=(A,B, \beta, T)$  be a relative Rota-Baxter group and  $\mathcal{K}=(K,L, \alpha,S )$ a trivial relative Rota-Baxter group. A central extension $\mathcal{E}$ of  $\mathcal{A}$ by   $\mathcal{K}$ is a short exact sequence of relative Rota-Baxter groups
	$$\mathcal{E} : \quad  {\bf 1} \longrightarrow (K,L, \alpha,S ) \stackrel{(i_1, i_2)}{\longrightarrow}  (H,G, \phi, R) \stackrel{(\pi_1, \pi_2)}{\longrightarrow} (A,B, \beta, T) \longrightarrow {\bf 1}$$ 
	such that $K \leq \Z^{\phi}_{R}(H)$ and $L \leq \Z(G) \cap \Ker(\phi)$.
\end{defn}

\begin{remark}
It is easy to see that a central extension $\mathcal{E}$ of relative Rota-Baxter groups induces two central extensions of groups
	$$\mathcal{E}_1 : \quad  1 \longrightarrow K \stackrel{i_1}{\longrightarrow}  H \stackrel{\pi_1 }{\longrightarrow} A \longrightarrow 1 \quad \textrm{and} \quad \mathcal{E}_2 : \quad  1 \longrightarrow L \stackrel{i_2}{\longrightarrow}  G \stackrel{\pi_2 }{\longrightarrow} B \longrightarrow 1.$$
\end{remark}

\begin{remark}
In view of Proposition \ref{rrb to slb homo}, it follows that every central extension $\mathcal{E}$ of relative Rota-Baxter groups induces an annhilator extension
	$$\mathcal{E}_{SLB}: \quad  {\bf 1} \longrightarrow K_{S} \stackrel{i_1}{\longrightarrow}  H_{R} \stackrel{\pi_1}{\longrightarrow} A_{T} \longrightarrow {\bf 1}$$
	of induced skew left braces.
\end{remark}

Let  $\mathcal{A}:=(A,B, \beta, T)$  be a relative Rota-Baxter group and  $\mathcal{K}:=(K,L, \alpha,S )$  a trivial relative Rota-Baxter group such that $K$ and $L$ are abelian groups. We can think of  $\mathcal{K}$ as a trivial $\mathcal{A}$-module (see \cite[Definition 3.12]{BRS2023}). Let $C(A \times B, K)$ be the group of all maps from $A \times B$ to $K$  which vanish on degenerate tuples. Similarly, let $C(A, L)$ denote the group of all maps from $A$ to $L$. Let 
\begin{align*}
	\Z_{RRB}^2(\mathcal{A}, \mathcal{K}):=\Bigg\{(\tau_1 ,\tau_2, \rho, \chi) \hspace{.1cm} \Big \vert  \hspace{.1cm}  \substack{ (\tau_1 ,\tau_2) \in \Z^2_{Gp}(A, K) \bigoplus \Z^2_{Gp}(B,L )\mbox{ and }  (\rho, \chi) \in  C(A \times B, K) \bigoplus C(A, L) \\ \mbox{ such that conditions } \eqref{RRBC1}, \eqref{RRBC2} \mbox{ and } \eqref{RRBC3} \mbox{ hold for all }\\  
	a_1, a_2 \in A \mbox{ and } b_1, b_2 \in B}  \Bigg\},
\end{align*}
where
\begin{eqnarray}
	\rho(a_1, b_1 b_2)  \,  (\rho(\beta_{b_2}(a_1), b_1))^{-1} \,(\rho(a_1,b_2))^{-1} &=&1, \label{RRBC1}\\
	\rho(a_1 a_2, b_1) \,(\rho(a_1,b_1))^{-1}(\rho(a_2, b_1))^{-1} \tau_1(a_1, a_2) \, (\tau_1(\beta_{b_1}(a_1), \beta_{b_1}(a_2)))^{-1}&=&1, \label{RRBC2} \\
	S\big(\rho(a_2, T(a_1)) \,\tau_1(a_1,\beta_{T(a_1)}(a_2)) \big) \, (\tau_2(T(a_1), T(a_2)))^{-1} \,(\partial^1(\chi)(a_1, a_2))^{-1}&=&1.\label{RRBC3}
\end{eqnarray}
It follows from \cite[ Proposition 3.8 and Proposition 3.10]{BRS2023} that every $(\tau_1 ,\tau_2, \rho, \chi) \in \Z^2_{RRB}(\mathcal{A}, \mathcal{K})$ defines a relative Rota-Baxter group $(A \times_{\tau_1} K, B \times_{\tau_2} L, \phi, R)$, where 
\begin{enumerate}
	\item the action $\phi: B \times_{\tau_2} L \longrightarrow \Aut(A \times_{\tau_1} K )$ is defined by
	\begin{eqnarray}\label{phidefnprop}
		\phi_{(b,l)}(a,k) &= & \big(\beta{_{b}(a)}, ~\rho(a,b)\;k\big),
	\end{eqnarray} 
	\item the relative Rota-Baxter operator   $R:  A \times_{\tau_1} K \longrightarrow B \times_{\tau_2} L$ is given by
	\begin{eqnarray}\label{Rformula}
		R(a,k) &= & \big(T(a), ~\chi(a) \,S(k)\big)
	\end{eqnarray} 
\end{enumerate}
for $(a,k) \in A \times_{\tau_1} K$ and $(b,l) \in B \times_{\tau_2} L$.

Let $\B^2_{RRB}(\mathcal{A}, \mathcal{K})$ denote the subgroup of $\Z^2_{RRB}(\mathcal{A}, \mathcal{K})$  comprising elements $(\tau_1, \tau_2, \rho, \chi)$ that are associated with maps $\theta_1: A \rightarrow K$ and $\theta_2: B \rightarrow L$ such that the following conditions are satisfied:
\begin{align}
	\tau_1(a_1, a_2) =&\; \partial^1(\theta_1)(a_1, a_2),\label{t1}\\
	\tau_2(b_1, b_2) =&\;   \partial^1(\theta_2)(b_1, b_2), \label{t2}\\
	\rho(a_1, b_1)= &\lambda_1(a_1, b_1):=  \theta_1(a_1) \,(\theta_1(\beta_{b_1}(a_1)))^{-1}, \label{t3}\\
	\chi(a_1)=& \lambda_2(a_1):= S(\theta_1(a_1)) \,(\theta_2(T(a_1)))^{-1}\label{t4}
\end{align}
for all $a_1, a_2 \in A$ and $b_1, b_2 \in B.$ Then the second cohomology of the relative Rota-Baxter group $\mathcal{A}$ with coefficients in $\mathcal{K}$ is defined by 
$$\Ho^2_{RRB}(\mathcal{A}, \mathcal{K}):=\Z^2_{RRB}(\mathcal{A},  \mathcal{K})/ \B^2_{RRB}(\mathcal{A},  \mathcal{K}).$$

Let $\CExt_{RRB}(\mathcal{A}, \mathcal{K})$ denote the set of equivalence classes of all central extensions of $\mathcal{A}$ by  $\mathcal{K}$.  Then we have the following result from \cite[Theorem 5.1]{BRS2023}

\begin{thm}
	The map $\Psi_3: \Ho^2_{RRB}(\mathcal{A}, \mathcal{K}) \longrightarrow \CExt_{RRB}(\mathcal{A}, \mathcal{K})$ defined by 
	$$ \Psi_3 \big( [(\tau_1, \tau_2, \rho, \chi)] \big)=[\mathcal{E}(\tau_1, \tau_2, \rho, \chi)]$$
	is a bijection, where $\mathcal{E}(\tau_1, \tau_2, \rho, \chi)$ represents the extension
	$$ \quad {\bf 1} \longrightarrow (K,L, \alpha,S ) \stackrel{(i_1, i_2)}{\longrightarrow}  (A \times_{\tau_1} K,B \times_{\tau_2} L, \phi, R) \stackrel{(\pi_1, \pi_2)}{\longrightarrow} (A,B, \beta, T) \longrightarrow {\bf 1}$$ 
	and the maps $\phi$ and $R$ are defined in \eqref{phidefnprop} and \eqref{Rformula}, respectively.
\end{thm}

\subsection{Connection between the three cohomology groups}
The connection between the second cohomology of skew left braces and the second cohomology of groups has been investigated in \cite{TV23}.  Let  $\mathcal{A}=(A,B, \beta, T)$  be a relative Rota-Baxter group and  $\mathcal{K}=(K,L, \alpha,S )$  a trivial relative Rota-Baxter group such that $K$ and $L$ are abelian groups, and  $\mathcal{K}$ is viewed as a trivial $\mathcal{A}$-module. It follows from \cite[Proposition 4.4]{BRS2023} that the map $\Pi_1: \Ho^2_{RRB}(\mathcal{A}, \mathcal{K}) \longrightarrow \Ho^2_{SLB}(A_T,K_S)$ defined by 
$$ \Pi_1 \big( [(\tau_1, \tau_2, \rho, \chi)]\big)=[(\tau_1, \tau^{(\beta, T)}_1\; \rho^T)]$$
is a group homomorphism, where
\begin{align*}
	\tau^{(\beta, T)}_1(a_1, a_2)=&\; \tau_1(a_1, \beta_{T(a_1)}(a_2)),\\
	\rho^T(a_1, a_2)=&\; \rho(a_2, T(a_1))
\end{align*}
for all $a_1, a_2 \in A.$ Now $[(\tau_1, \tau_2, \rho, \chi)] \in \ker (\Pi_1)$ if there exists a map $\theta: A \rightarrow K$ such that
\begin{align}
	\tau_1(a_1, a_2)=&\; \theta(a_2)\; \theta(a_1 \cdot a_2)^{-1}\; \theta(a_1),\label{ker1}\\
	\tau_1(a_1, \beta_{T(a_1)}(a_2)) \;  \rho(a_2, T(a_1))=&\;  \theta(a_2)\; \theta(a_1 \circ_T a_2)^{-1}\; \theta(a_1) \label{ker2}
\end{align}
for all $a_1, a_2 \in A.$ Using the value of $\tau_1$ from  \eqref{ker1} in \eqref{ker2}, we obtain
$$ \theta(\beta_{T(a_1)}(a_2))\; \theta(a_1 \circ_T a_2)^{-1}\; \theta(a_1) \;  \rho(a_2, T(a_1))= \theta(a_2)\; \theta(a_1 \circ_T a_2)^{-1}\; \theta(a_1).$$
Since $K$ is abelian, this further gives
$$\rho(a_2, T(a_1))= \theta(a_2) \; \theta(\beta_{T(a_1)}(a_2))^{-1}$$
for all $a_1, a_2 \in A$. Hence,
\begin{align*}
	\ker (\Pi_1)=\Bigg\{[(\tau_1 ,\tau_2, \rho, \chi)] \in \Ho^2_{RRB}(\mathcal{A}, \mathcal{K}) \hspace{.1cm} \Big \vert  \hspace{.1cm}  \substack{ \tau_1(a_1, a_2)=\; \theta(a_2)\; \theta(a_1 \cdot a_2)^{-1}\; \theta(a_1)  \mbox{ and }  \rho(a_2, T(a_1))= \; \theta(a_2) \; \theta(\beta_{T(a_1)}(a_2))^{-1} \\ \mbox{ for some map  } \theta:A \to K}  \Bigg\}.
\end{align*}

\begin{remark}
	If $\mathcal{A}$ is a trivial relative Rota-Baxter group, then $\ker (\Pi_1)$ consist of $[(\tau_1 ,\tau_2, \rho, \chi)]$ such that $\tau_1 \in \B^2_{Gp}(A, K)$ and $\rho^T$ is a trivial map.
\end{remark}

It is easy to see that the following  maps are group homomorphisms:
\begin{enumerate}
	\item  $\Pi_2: \Ho^2_{RRB}(\mathcal{A}, \mathcal{K}) \longrightarrow \Ho^2_{Gp}(A,K)$  defined by $\Pi_2 \big([(\tau_1, \tau_2, \rho, \chi)] \big)=[\tau_1].$
	\item $\Pi_3: \Ho^2_{RRB}(\mathcal{A}, \mathcal{K}) \longrightarrow \Ho^2_{Gp}(B,L)$  defined by 	$\Pi_3 \big([(\tau_1, \tau_2, \rho, \chi)] \big)=[\tau_2].$
	\item $\Pi_4: \Ho^2_{RRB}(\mathcal{A}, \mathcal{K}) \longrightarrow \Ho^2_{Gp}(A,K) \bigoplus \Ho^2_{Gp}(B,L)$ defined by  $\Pi_4 \big([(\tau_1, \tau_2, \rho, \chi)] \big)=([\tau_1], [\tau_2]).$
\end{enumerate}

Further, the kernels of $\Pi_2$, $\Pi_3$ and $\Pi_4$ are as follows: 
\begin{enumerate}
	\item $\ker (\Pi_2)= \big\{ [(\tau_1 ,\tau_2, \rho, \chi)] \in \Ho^2_{RRB}(\mathcal{A}, \mathcal{K}) \mid \tau_1 \in \B^2_{Gp}(A, K)   \big\}.$
	\item $\ker (\Pi_3)= \big\{ [(\tau_1 ,\tau_2, \rho, \chi)] \in \Ho^2_{RRB}(\mathcal{A}, \mathcal{K}) \mid \tau_2 \in \B^2_{Gp}(B, L)   \big\}.$
	\item $\ker (\Pi_4)= \big\{ [(\tau_1 ,\tau_2, \rho, \chi)] \in \Ho^2_{RRB}(\mathcal{A}, \mathcal{K}) \mid (\tau_1, \tau_2) \in \B^2_{Gp}(A, K) \bigoplus \B^2_{Gp}(B, L)  \big\}.$
\end{enumerate}
\medskip

  \section{Hochschild-Serre exact sequence for relative Rota-Baxter groups}\label{Hochschild-Serre sequence}

Given maps $f_1, \ldots, f_n:X \longrightarrow Y$, let $(f_1, \ldots, f_n): X \longrightarrow Y \times \cdots \times Y$ be the map given by $(f_1, \ldots, f_n)(x)= \big(f_1(x), \ldots, f_n(x) \big)$ for $x \in X$.

\begin{prop}\label{cohomology with product coeff}
Let $\mathcal{A} =(A,B,\beta,T)$ be a relative Rota-Baxter group and $\mathcal{K}_1,\dots ,\mathcal{K}_n$ be trivial $\mathcal{A}$-modules, where $\mathcal{K}_i =(K_i,L_i,\alpha_i,S_i)$. Then the map 
$$\Psi: \prod_{i=1}^n \Ho^2_{RRB}\left(\mathcal{A},\mathcal{K}_i\right) \longrightarrow \Ho ^2_{RRB}\left(\mathcal{A},\mathcal{K}_1\times\cdots\times \mathcal{K}_n\right)$$
given by
$$\big([(\tau_{1_1},\tau_{2_1},\rho_1,\chi_1)],\dots,[(\tau_{1_n},\tau_{2_n},\rho_n,\chi_n)] \big) \longmapsto \big[\left((\tau_{1_1},\dots,\tau_{1_n}),(\tau_{2_1},\dots,\tau_{2_n}),(\rho_1,\dots,\rho_n),(\chi_1,\dots,\chi_n)\right) \big],$$
is an isomorphism of cohomology groups. 
\end{prop}
\begin{proof}
Let $\mathcal{K}=\mathcal{K}_1\times\cdots\times \mathcal{K}_n$, $K= K_1 \times \cdots \times K_n$ and $L= L_1 \times \cdots \times L_n$. Define
$$ \psi : \prod_{i=1}^n \Z^2_{RRB}(\mathcal{A},\mathcal{K}_i) \longrightarrow \Ho^2_{RRB}(\mathcal{A},\mathcal{K})$$
by 
$$\psi \big((\tau_{1_1},\tau_{2_1},\rho_1,\chi_1),\dots,(\tau_{1_n},\tau_{2_n},\rho_n,\chi_n)\big) =\big[\left((\tau_{1_1},\dots,\tau_{1_n}),(\tau_{2_1},\dots,\tau_{2_n}),(\rho_1,\dots,\rho_n),(\chi_1,\dots,\chi_n)\right)\big]
$$
for $(\tau_{1_i},\tau_{2_i},\rho_i,\chi_i) \in  \Z^2_{RRB}(\mathcal{A},\mathcal{K}_i)$.
Clearly,  we have 
$$ \big((\tau_{1_1},\dots,\tau_{1_n}),(\tau_{2_1},\dots,\tau_{2_n}),(\rho_1,\dots,\rho_n),(\chi_1,\dots,\chi_n) \big)\in \Z^2_{RRB}(\mathcal{A},\mathcal{K}),$$ and $\psi$ is a well-defined group homomorphism. Suppose that 
$$\mu=\big((\partial^1(\theta_{1_1}),\partial^1(\theta_{2_1}),\lambda_{1_1},\lambda_{2_1}),\dots , (\partial^1(\theta_{1_n}),\partial^1(\theta_{2_n}),\lambda_{1_n},\lambda_{2_n})\big)\in \prod_{i=1}^n \B_{RRB}(\mathcal{A},\mathcal{K}_i).$$
Define $\theta_1: A\to K$,  $\theta_2: B\to L$, $\lambda_1: A\times B\to K$ and $\lambda_2: A\to L$ by 
\begin{eqnarray*}
\theta_1(a) &=& (\theta_{1_1}(a),\dots,\theta_{1_n}(a)),\\
\theta_2(b) &=& (\theta_{2_1}(b),\dots,\theta_{2_n}(b)),\\
\lambda_1(a,b) &=& (\lambda_{1_1}(a,b),\dots,\lambda_{1_n}(a,b)),\\\
\lambda_2(a) &=& (\lambda_{2_1}(a),\dots,\lambda_{2_n}(a))
\end{eqnarray*}
for $a \in A$ and $b \in B$.  Since 
\begin{eqnarray*}
	\psi\left(\mu\right)&=& \big[((\partial^1(\theta_{1_1}),\dots,\partial^1(\theta_{1_n})),(\partial^1(\theta_{2_1}),\dots,\partial^1(\theta_{2_n})),(\lambda_{1_1},\dots,\lambda_{1_n}),(\lambda_{2_1},\dots,\lambda_{2_n})) \big]\\
	&=& \big[(\partial^1\theta_1,\partial^1\theta_2,\lambda_1,\lambda_2) \big]\\
	&=& 1,
\end{eqnarray*}
it follows that $\psi$ induces the homomorphism $\Psi$ as given in the statement of the proposition. We now define an inverse to $\Psi$. Let $ \phi: \Z^2_{RRB}(\mathcal{A},\mathcal{K}) \longrightarrow \prod_{i=1}^n \Ho^2_{RRB}(\mathcal{A},\mathcal{K}_i)$ be given by
$$
\phi (\tau_1,\tau_2,\rho,\chi)= \big([\left(\tau_{1_1},\tau_{2_1}, \rho_1, \chi_1\right)],\dots, [\left(\tau_{1_n},\tau_{2_n}, \rho_n, \chi_n\right)] \big),
$$
where $\tau_1= (\tau_{1_1}, \ldots, \tau_{1_n})$,  $\tau_2= (\tau_{2_1}, \ldots, \tau_{2_n})$, $\rho=(\rho_1, \ldots, \rho_n)$ and $\chi= (\chi_1, \ldots, \chi_n)$. Similar arguments as above show that $\phi$ is well-defined and induces a homomorphism
$$\Phi: \Ho^2_{RRB}(\mathcal{A},\mathcal{K})\longrightarrow \prod_{i=1}^n \Ho^2_{RRB}(\mathcal{A},\mathcal{K}_i), $$
which, by construction, is the inverse of $\Psi$. 
\end{proof}

Suppose that
 \begin{equation}\label{V1}
 	\mathcal{E} : \quad  {\bf 1} \longrightarrow (K,L, \alpha,S ) \stackrel{(i_1, i_2)}{\longrightarrow}  (H,G, \phi, R) \stackrel{(\pi_1, \pi_2)}{\longrightarrow} (A,B, \beta, T) \longrightarrow {\bf 1}
 \end{equation}	
is a central extension of relative Rota-Baxter groups and $(M,N,\gamma^{0},U^0)$ a trivial module over $(A,B,\beta,T)$.
 For brevity, we write $\mathcal{K}=(K,L,\alpha,S)$, $\mathcal{H}=(H,G,\phi,R)$, $\mathcal{A}=(A,B,\beta,T)$, $\mathcal{M}=(M,N,\gamma^{0},U^0)$, and define the following maps.
 \begin{enumerate}
 	\item Restriction homomorphism: A morphism $(f_1,f_2)\in \Hom_{RRB}{(\mathcal{H}, \mathcal{M})}$ induces a morphism $(f_1|_K,f_2|_L)\in \Hom_{RRB}{(\mathcal{K}, \mathcal{M})}$ by restricting the map $f_1$ to $K$ and the map $f_2$ to $L$. This assignment gives a group homomorphism $$\Res: {\Hom_{RRB}{(\mathcal{H}, \mathcal{M})}}{\longrightarrow}{\Hom_{RRB}{(\mathcal{K}, \mathcal{M})}},$$
which we refer as the {\it restriction homomorphism}.
 	\item Transgression homomorphism: 	Let $(\tau_1,\tau_2,\rho,\chi)$ be a 2-cocycle corresponding to the central extension $\eqref{V1}$. For a morphism $(g_1,g_2)\in \Hom_{RRB}{(\mathcal{K}, \mathcal{M})}$, we have $(g_1\tau_1,g_2\tau_2,g_1\rho,g_2\chi) \in \Z^2_{RRB}{(\mathcal{A}, \mathcal{M})}$. Routine calculations show that the map $${(g_1,g_2)} \longmapsto [(g_1\tau_1,g_2\tau_2,g_1\rho,g_2\chi)]$$ depends only on the cohomology class of $(\tau_1,\tau_2,\rho,\chi)$, and hence gives the group homomorphism $$\Tra : {\Hom_{RRB}{(\mathcal{K}, \mathcal{M})}}{\longrightarrow}{\Ho^2_{RRB}{(\mathcal{A}, \mathcal{M})}},$$
which we refer as the {\it transgression homomorphism}.
 	\item  Inflation homomorphism: Let $(\psi_1,\psi_2)\in \Hom_{RRB}{(\mathcal{A}, \mathcal{M})}$ be a morphism. Then $(\psi_1\pi_1,\psi_2\pi_2)\in \Hom_{RRB}{(\mathcal{H}, \mathcal{M})}$ and we obtain a group homomorphism $$\Inf: {\Hom_{RRB}{(\mathcal{A}, \mathcal{M})}}{\longrightarrow}{ \Hom_{RRB}{(\mathcal{H}, \mathcal{M})}},$$
which we refer as the {\it inflation homomorphism}.	We also have a similar group homomorphism at the level of cohomology
 	$$\overline{\Inf}: {\Ho^2_{RRB}{(\mathcal{A}, \mathcal{M})}}{\longrightarrow}{ \Ho^2_{RRB}{(\mathcal{H}, \mathcal{M})}}$$
given by
 	$$[(\tau_1,\tau_2,\rho,\chi)] \longmapsto [(\tau_1(\pi_1\times\pi_1),\tau_2(\pi_2\times\pi_2 ),\rho(\pi_1\times\pi_2),\chi\pi_1)],$$
which we again refer as the inflation homomorphism.
 \end{enumerate}

We now present the main result of this section.

 \begin{thm}{\label{HSS}}
 	Suppose that \begin{align}\label{V2}
 		\mathcal{E} : \quad  {\bf 1} \longrightarrow (K,L, \alpha,S ) \stackrel{(i_1, i_2)}{\longrightarrow}  (H,G, \phi, R) \stackrel{(\pi_1, \pi_2)}{\longrightarrow} (A,B, \beta, T) \longrightarrow {\bf 1}
 	\end{align}	
 	is a central extension of relative Rota-Baxter groups and $(M,N,\gamma^{0},U^0)$ is a trivial module over $(A,B,\beta,T)$. Then
 	\begin{small}
 		\begin{align}{\label{V3}}
 			{\bf 1} \longrightarrow {\Hom_{RRB}{(\mathcal{A},\mathcal{M})}} \stackrel{\Inf}{\longrightarrow} { \Hom_{RRB}{(\mathcal{H},\mathcal{M})}} \stackrel{\Res}{\longrightarrow} {\Hom_{RRB}{(\mathcal{K},\mathcal{M})}} \stackrel{\Tra} \longrightarrow{\Ho^2_{RRB}{(\mathcal{A}, \mathcal{M})}} \stackrel{\overline{\Inf}}{\longrightarrow} {\Ho^2_{RRB}{(\mathcal{H},\mathcal{M})}}
 		\end{align}	
 	\end{small}
 	is an exact sequence of groups.	
 \end{thm}
 
 \begin{proof}
 	Recall from \cite[Theorem 3.18]{BRS2023} that for a given set-theoretic section $(s_H, s_G)$ of $\mathcal{E}$, the quadruple $(\tau_1, \tau_2, \rho, \chi)$ defined below gives an element of $\Ho^2_{RRB}(\mathcal{A}, \mathcal{K})$, where
 	\begin{eqnarray}
\tau_1(a_1, a_2) & = &  s_H(a_1 a_2)^{-1}s_H(a_1)s_H(a_2)\label{c1},\\
\tau_2(b_1, b_2) &=&  s_G(b_1 b_2)^{-1}s_G(b_1)s_G(b_2)\label{c2}, \\
\rho(a_1, b_1)& =& s_H(\beta_{b_1}(a_1))^{-1} \phi_{s_G(b_1)}(s_H(a_1)) \label{c3}, \\
 \chi(a_1) &=& s_G(T(a_1))^{-1} R(s_H(a_1)) \label{c4} 
 	\end{eqnarray}
 	for all $a_1, a_2 \in A$ and $b_1, b_2 \in B$.

Exactness at $\Hom_{RRB}{(\mathcal{A},\mathcal{M})}$: By exactness of $\eqref{V2}$, the maps $\pi_1,\pi_2$ are surjective, and hence the homomorphism  $\Inf: \Hom_{RRB}{(\mathcal{A},\mathcal{M})} \longrightarrow { \Hom_{RRB}{(\mathcal{H},\mathcal{M})}}$ is injective.
 	\par
 	
Exactness at $\Hom_{RRB}{(\mathcal{H},\mathcal{M})}$: For $(\psi_1,\psi_2)\in \Hom_{RRB}{(\mathcal{A},\mathcal{M})}$, we have  $(\psi_1\pi_1,\psi_2\pi_2)\in \Hom_{RRB}{(\mathcal{H},\mathcal{M})}$. Since $\ker(\pi_1)=K$ and  $\ker(\pi_2)=L$, we have that $\Res(\psi_1\pi_1,\psi_2\pi_2)$ is trivial, and hence $\im(\Inf) \subseteq \ker(\Res)$. If $(f_1,f_2)\in \ker(\Res)$, then both $f_1|_K$ and $f_2|_L$ are trivial. For $h\in H$ and $g \in G$, let $\overline{h}\in A \cong H/K$ and $\overline{g}\in B \cong G/L$ denote their images in the quotient. Define $\psi_1: A \to M$ and $\psi_2:B \to N$ by
 	$$ \psi_1(\overline{h})= f_1(h) \quad \textrm{and} \quad  \psi_2(\overline{g})= f_2(g).$$
 	Then $(\psi_1,\psi_2) \in \Hom_{RRB}{(\mathcal{A},\mathcal{M})}$ and $\Inf(\psi_1,\psi_2) )= (f_1, f_2)$, which is desired.
 	\par 
 	
Exactness at $\Hom_{RRB}{(\mathcal{K},\mathcal{M})}$: For $(f_1,f_2)\in \Hom_{RRB}{(\mathcal{H},\mathcal{M})}$, let $\Res(f_1, f_2)=(g_1, g_2)$. Let $(\tau_1,\tau_2,\rho,\chi)$ be a 2-cocycle corresponding to $\eqref{V2}$, and $(s_H, s_G)$ a set-theoretic section to $\mathcal{E}$. Then $\Tra(\Res(f_1,f_2))= [(g_1\tau_1,g_2\tau_2,g_2\rho,g_2\chi)]$. Now,  for $a_1, a_2 \in A$ and $b_1, b_2 \in B$, we have
 	\begin{equation*}
 		g_1\tau_1(a_1,a_2) = f_1(s_H(a_1a_2)^{-1}s_H(a_1)s_H(a_2))
 	\end{equation*} 
 	and
 	\begin{equation*}
 		g_2\tau_2(b_1,b_2) = f_2((s_G(b_1b_2)^{-1})s_G(b_1)s_G(b_2)).
 	\end{equation*}
 	From \eqref{c3}, we have
 	\begin{eqnarray*}
 		g_1\rho(a_1,b_1) &=& f_1\phi_{s_G(b_1)}(s_H(a_1)) \, \, f_1(s_H(\beta_{b_1}(a_1)))^{-1}\\
 		&=& \gamma_{\eta(s_G(b_1))}^{0} f_1(s_H(a_1)) \, \,f_1(s_H(\beta_{b_1}(a_1)))^{-1}\\
 		&=& f_1s_H(a_1) \,\, f_1(s_H(\beta_{b_1}(a_1)))^{-1}, \quad \textrm{since}~ \gamma^{0}~\textrm{is trivial}
 	\end{eqnarray*}
 	whereas \eqref{c4} gives
 	\begin{eqnarray*}
 		g_2\chi(a_1) &=& f_2R(s_H(a_1)) \, \, f_2s_G(Ta_1)^{-1}  \\
 		&=& U^0(f_1s_H(a_1))\, \, f_2s_G(T(a_1))^{-1}.
 	\end{eqnarray*}
 	This implies that $(g_1\tau_1,g_2\tau_2,g_1\rho,g_2\chi) $ is a coboundary, and hence $\im(\Res) \subseteq \ker(\Tra)$. Now, let $(g_1,g_2)\in \ker(\Tra)$. Since $(g_1\tau_1,g_2\tau_2,g_1\rho,g_2\chi) $ is a 2-coboundary, there exist $\theta_1 : A {\longrightarrow} M $ and $\theta_2 : B {\longrightarrow} N $ 
 	such that  
 	\begin{eqnarray*}
 		g_1(\tau_1(a_1,a_2)) &=& \theta_1(a_2) \, \theta_1(a_1a_2)^{-1} \, \theta_1(a_1),\\
 		g_2(\tau_2(b_1,b_2)) &=& \theta_2(b_2) \, \theta_2(b_1b_2)^{-1} \, \theta_2(b_1).
 	\end{eqnarray*}
 	for $a_1, a_2 \in A$ and $b_1, b_2 \in B$.	Define $r_1:H \longrightarrow M $ and	$r_2:G \longrightarrow N$ by
 	$$r_1(s_H(a)k)=\theta_1(a) \, g_1(k) \quad \textrm{and} \quad  r_2(s_G(b)l )=\theta_2(b) \, g_2(l).$$
 	 	It can be seen that $(r_1,r_2) \in \Hom_{RRB}({\mathcal{H}},\mathcal{M})$ and clearly $\Res(r_1, r_2)=(g_1, g_2)$, which proves that $\im(\Res) = \ker(\Tra)$.
 	\par 
Exactness at $ \Ho^2_{RRB}{(\mathcal{A},\mathcal{M})}$:  If $(g_1,g_2) \in \Hom_{RRB}({\mathcal{K}},\mathcal{M})$, then
 	$$ \overline{\Inf}(\Tra(g_1,g_2)) = \big[ (g_1\tau_1(\pi_1\times\pi_1),g_2\tau_2(\pi_2\times\pi_2),g_1\rho(\pi_1\times\pi_2),g_2\chi\pi_1) \big].$$
 	Note that $s_H(\pi_1(x))x^{-1} \in \ker(\pi_1)=K$ for each $x \in H$. Thus, we can define  $\theta_1 : H {\longrightarrow} M $ by $ \theta_1 (x)= g_1(s_H(\pi_1(x))x^{-1})$ for $x \in H$. Consequently, we have
 	\begin{eqnarray*}
 		g_1\tau_1(\pi_1\times\pi_1)(x,y) &=& g_1\tau_1(\pi_1(x),\pi_1(y))\\
 		&=&  g_1\big(s_H(\pi_1(xy))^{-1}s_H(\pi_1(x))s_H(\pi_1(y))\big)\\
 		&=&  g_1\big((xy)s_H(\pi_1(xy))^{-1}s_H(\pi_1(x))s_H(\pi_1(y))(xy)^{-1})\big)\\ 
 		&=& g_1\big((s_H(\pi_1(xy))(xy)^{-1})^{-1} \, s_H(\pi_1(x)) \, s_H(\pi_1(y))y^{-1}x^{-1}\big) \\
 		&=& g_1\big((s_H(\pi_1(xy))(xy)^{-1})^{-1} \, s_H(\pi_1(x))x^{-1} \, s_H(\pi_1(y))y^{-1}\big) \\
 		&=& g_1\big(s_H(\pi_1(xy))(xy)^{-1} \big)^{-1} \, g_1\big(s_H(\pi_1(x))x^{-1}\big) \, g_1\big(s_H(\pi_1(y))y^{-1}\big) \\
 		&=& \theta_1(xy)^{-1}  \, \theta_1(x)\, \theta_1(y)\\
 		&=& \partial^1\theta_1(x,y)
 	\end{eqnarray*}
 	for all $x, y \in H$. Similarly, noting that $s_G(\pi_2(x))x^{-1} \in \ker(\pi_2)=L$ for each $x \in G$ and  defining  $\theta_2 : G {\longrightarrow} N $ by $ \theta_2 (x)= g_2(s_G(\pi_2(x))x^{-1})$,  we obtain $g_2\tau_2(\pi_2\times\pi_2) = \partial^1\theta_2 $. Also, for $x \in H$ and $y \in G$, we have
 	\begin{eqnarray*}
 		g_1\rho(\pi_1\times\pi_2)(x,y) &=& g_1\rho(\pi_1(x),\pi_2(y))\\
 		&=& g_1\big((s_H(\beta_{\pi_2(y)}(\pi_1(x))))^{-1} \,\phi_{ s_G(\pi_2(y))}(s_H(\pi_1(x)))\big),  \textrm{ using}~\eqref{c3}\\
 		&=& g_1 \big((s_H(\beta_{\pi_2(y)}(\pi_1(x))))^{-1} \,\phi_{s_G(\pi_2(y))}(s_H(\pi_1(x))x^{-1}x) \big)\\
 		&=& g_1 \big((s_H(\beta_{\pi_2(y)}(\pi_1(x))))^{-1} \,\phi_{s_G(\pi_2(y))}(s_H(\pi_1(x))x^{-1})\phi_{s_G(\pi_2(y))}(x) \big)\\
 		&=& g_1 \big((s_H(\beta_{\pi_2(y)}(\pi_1(x))))^{-1} \, (s_H(\pi_1(x))x^{-1}) \,(\phi_{s_G(\pi_2(y))}(x)) \big),\\
 		&& \textrm{since}~\phi~ \textrm{acts trivially on}~K\\
 		&=& g_1 \big((s_H(\beta_{\pi_2(y)}(\pi_1(x))))^{-1} \,(s_H(\pi_1(x))x^{-1}) \,(\phi_{ly}(x)) \big),~~ \textrm{for some}~l~ \textrm{in}~L\\
 		&=& g_1 \big((s_H(\beta_{\pi_2(y)}(\pi_1(x))))^{-1} \, (s_H(\pi_1(x))x^{-1}) \, (\phi_{y}(x)) \big),~~ \textrm{since}~l\in\ker(\phi)\\
 		&=& g_1 \big((\phi_{y}(x))(s_H(\beta_{\pi_2(y)} (\pi_1(x))))^{-1} \,(s_H(\pi_1(x))x^{-1}) \, (\phi_{y}(x))(\phi_{y}(x))^{-1} \big)\\
 		&=& g_1 \big((s_H(\beta_{\pi_2(y)}(\pi_1(x)))(\phi_{y}(x))^{-1})^{-1} \,(s_H(\pi_1(x))x^{-1} )\big)\\
 		&=& g_1 \big( (s_H(\pi_1(\phi_y(x)))(\phi_{y}(x))^{-1})^{-1} \, \,(s_H(\pi_1(x))x^{-1} )\big),\\
		&& \textrm{ since $(\pi_1, \pi_2)$ is a homomorphism }\\
 		&=& \theta_1(\phi_y(x))^{-1}\theta_1(x)
 	\end{eqnarray*}
 	and
 	\begin{eqnarray*}
 		g_2 (\chi(\pi_1(x))) &=& g_2 \big(s_G(T(\pi_1(x)))^{-1}~R(s_H(\pi_1(x))) \big),  \textrm{ using}~\eqref{c4}\\
 		&=& g_2 \big(s_G(\pi_2(R(x)))^{-1}~R(s_H(\pi_1(x))) \big), \textrm{ since  $T \pi_1=\pi_2 R$ } \\
 		&=& g_2 \big(R(x)s_G(\pi_2(R(x)))^{-1}~R(s_H(\pi_1(x)))R(x)^{-1} \big)~ \textrm{ conjugating by $R(x)$ }\\
 		&=& g_2 \big((s_G(\pi_2(R(x)))R(x)^{-1})^{-1}~R(s_H(\pi_1(x)))R(x)^{-1} \big)\\
 		&=& g_2 \big((s_G(\pi_2(R(x)))R(x)^{-1})^{-1}~R(s_H(\pi_1(x))x^{-1}x)~R(x)^{-1} \big)\\
 		&=& g_2 \big((s_G(\pi_2(R(x)))R(x)^{-1})^{-1}~R(s_H(\pi_1(x))x^{-1}\phi_{R(s_H(\pi_1(x))x^{-1})}(x))R(x)^{-1} \big),\\
 		&&\textrm{since}~R|_K=S,~S~\textrm{maps}~K~\textrm{to}~L~\textrm{and}~\phi|_L~~\textrm{is trivial}\\
 		&=& g_2 \big((s_G(\pi_2(R(x)))~R(x)^{-1})^{-1}R(s_H(\pi_1(x))x^{-1})~R(x)~R(x)^{-1} \big), \textrm{ by definition of }R\\
 		&=& g_2(s_G(\pi_2(R(x)))R(x)^{-1})^{-1}~g_2(S(s_H(\pi_1(x))x^{-1})), \quad \textrm{since}~R|_K=S\\
 		&=& U^0(g_1(s_H(\pi_1(x))x^{-1}))~g_2(s_G(\pi_2(R(x)))R(x)^{-1})^{-1}, \textrm{ since}~(g_1,g_2)  \textrm{ is a homomorphism}\\
 		&=&U^0(\theta_1(x))~\theta_2(R(x))^{-1}.
 	\end{eqnarray*}

 	Thus, $\overline{\Inf}(\Tra(g_1,g_2))$ is a 1-coboundary, and we have $\im(\Tra) \subseteq \ker( \overline{\Inf})$. Now, let $\big[(\tau'_1,\tau'_2,\rho',\chi') \big] \in\ker(\overline{\Inf})$. Then there exist maps $\theta'_1: H\rightarrow M$ and $\theta'_2: G\rightarrow N$ such that
 	\begin{eqnarray}
 		\tau'_1(a_1,a_2) &=& \theta'_1(s_H(a_1)k_1)~ \big(\theta'_1(s_H(a_1a_2)k_1k_2\tau_1(a_1,a_2)) \big)^{-1}~ \theta'_1(s_H(a_2)k_2), \label{coc1}\\
 		\tau'_2(b_1,b_2) &=& \theta'_2(s_G(b_1)l_1)~ \big(\theta'_2(s_G(b_1b_2)l_1l_2\tau_2(b_1,b_2))\big)^{-1} ~\theta'_2(s_G(b_2)l_2), \label{coc2}\\
 		\rho'(a_1,b_1) &=& \theta'_1(s_H(a_1)k_1) ~\big(\theta'_1(\phi_{s_G(b_1)l_1}(s_H(a_1)k_1)) \big)^{-1}, \label{coc3}\\
 		\chi'(a_1) &=& U^0(\theta'_1(s_H(a_1)k_1))~\big(\theta'_2(R(s_H(a_1)k_1))\big)^{-1},\label{coc4}
 	\end{eqnarray}
 	for all $a_1,a_2\in A$, $b_1,b_2\in B$, $k_1,k_2\in K$ and $l_1,l_2\in L$. Consider the maps $g_1: K\to M$ and $g_2: L\to N$ given by  $g_1(k)= \theta'_1(k)$ and $g_2(l)= \theta'_2(l)$.  By setting $a_1=a_2=1$ in \eqref{coc1} (similarly $b_1=b_2=1$ in \eqref{coc2}),  we see that $g_1$ and $g_2$ are group homomorphisms. Also, consider the maps $\theta_1: A\to M$ and $\theta_2: B\to N$ given by $\theta_1 (a)= \theta'_1(s_H(a))$ and  $\theta_2(b)= \theta'_2(s_G(b))$. Taking $k_1=k_2=1$ and $l_1=l_2=1$ in \eqref{coc1}, \eqref{coc2},\eqref{coc3} and \eqref{coc4}, we get
 	\begin{eqnarray*}
 		\tau'_1(a_1,a_2) &=& \theta_1(a_1)~ \big(\theta_1(a_1a_2)\big)^{-1}~\theta_1(a_2)~\big(g_1(\tau_1(a_1,a_2)))^{-1},\\
 		\tau'_2(b_1,b_2)  &=& \theta_2(b_1)~\big(\theta_2(b_1b_2))^{-1}~\theta_2(b_2)~\big(g_2(\tau_2(b_1,b_2))\big)^{-1},\\
 		\rho'(a_1,b_1) &=& \theta_1(a_1)~\big(\theta_1(\beta_{b_1}(a_1))\big)^{-1}~\big(g_1(\rho(a_1,b_1))\big)^{-1}, \quad \textrm{using}~\eqref{c3}\\
 		\chi'(a_1) &=& U^0(\theta_1(a_1))~\big(\theta_2(T(a_1))\big)^{-1}~\big(g_2(\chi(a_1))\big)^{-1}, \quad \textrm{using}~\eqref{c4}.
 	\end{eqnarray*}
 	Moreover, for $a_1=1$ in \eqref{coc4}, we get $U^0g_1=g_2S$, and hence $(g_1,g_2)\in \Hom_{RRB}(\mathcal{K},\mathcal{M})$. This shows that 
 	$\big[ (\tau'_1,\tau'_2,\rho',\chi') \big] \in \IM(\Tra)$, and the proof of the theorem is complete.
 \end{proof}	
 \medskip

 \section{Schur multiplier of relative Rota-Baxter groups}\label{sec Schur multiplier and cover}

We begin with our main definition.
 
 \begin{defn}
The Schur multiplier of a relative Rota-Baxter group $\mathcal{A}=(A,B,\beta,T)$ is defined to be the group $M_{RRB}(\mathcal{A})=\Ho^2_{RRB}{(\mathcal{A},\mathcal{C})}$, where $\mathcal{C}=(\mathbb{C}^{\times},\mathbb{C}^{\times},\alpha^0,S^0)$ is the relative Rota-Baxter group such that $\alpha^0$ is trivial and $S^0$ is the identity map.
 \end{defn}
  
Note that we could have taken $S^0$ to be any automorphism of $\mathbb{C}^{\times}$ since each such relative Rota Baxter group is isomorphic to $\mathcal{C}$.  A relative Rota-Baxter group $(A,B,\beta,T)$ is said to be {\it finite} if both $A$ and $B$ are finite. Building upon the arguments for the group case, we prove the following result.

 \begin{thm}\label{VI1}
 	If $\mathcal{A} = (A,B,\beta,T)$ is a finite relative Rota-Baxter group, then the exponent of $M_{RRB}(\mathcal{A})$ divides $|A||B|$.
 \end{thm}
 \begin{proof}
 	Let $|A|=m$, $|B|=n$ and  $(\tau_1,\tau_2,\rho,\chi)$ represent an element of $M_{RRB}(\mathcal{A})$. Let
 	\begin{align}
 		\mathcal{E} : \quad  {\bf 1} \longrightarrow (\mathbb{C}^{\times},\mathbb{C}^{\times},\alpha^0,S^0) \stackrel{(i_1, i_2)}{\longrightarrow}  (H,G, \phi, R) \stackrel{(\pi_1, \pi_2)}{\longrightarrow} (A,B, \beta, T) \longrightarrow {\bf 1}
 	\end{align}	
 	be a central extension corresponding to $(\tau_1,\tau_2,\rho,\chi)$, and $(s_H,s_G)$ a set-theoretic section to $\mathcal{E}$. Then $\tau_1(a_1, a_2)= s_H(a_1 a_2)^{-1}s_H(a_1)s_H(a_2)$ for $ a_1, a_2 \in A$ and 	$\tau_2(b_1, b_2)= s_G(b_1 b_2)^{-1}s_G(b_1)s_G(b_2)$ for  $b_1, b_2 \in B$.  Also, we have the group 2-cocycle conditions
 	\begin{eqnarray*}
 		\tau_1(a_1a_2,a_3)^{-1}\tau_1(a_1,a_2a_3)\tau_1(a_2,a_3)&=&\tau_1(a_1,a_2),\\
 		\tau_2(b_1b_2,b_3)^{-1}\tau_2(b_1,b_2b_3)\tau_2(b_2,b_3)&=&\tau_2(b_1,b_2),
 	\end{eqnarray*}
 	for all $a_1,a_2,a_3\in A$ and $b_1,b_2,b_3 \in B$.
	For each $x \in A$ and $y \in B$, let
 	\begin{equation}
 		\label{VI11}
 		A(x)=\prod_{a\in A}\tau_1(x,a) \quad \textrm{and} \quad
 		B(y)=\prod_{b\in B}\tau_2(y,b).
 	\end{equation}
 	By taking product over all $a_3 \in A$ and all $b_3\in B$, we obtain
 	\begin{equation*}
 		A(a_1 a_2)^{-1}A(a_1)A(a_2)=\tau_1(a_1,a_2)^m
 	\end{equation*}
 	and similarly
 	\begin{equation}\label{Mcond}
 		B(b_1 b_2)^{-1} B(b_1) B(b_2)=\tau_2(b_1,b_2)^n.
 	\end{equation}
 	Let $\theta_1: A\to \mathbb{C}^{\times}$ and $\theta_2: B\to \mathbb{C}^{\times}$ be maps given by
 	$$\theta_1(x)=A(x)^n \quad \textrm{and} \quad \theta_2(y)=B(y)^m.$$ Then we have $\theta_1(1)=1$ and $\theta_2(1)=1$. Further,
\begin{equation}\label{theta1 theta2 coboundary}
 \theta_1(a_2)\theta_1(a_1a_2)^{-1}\theta_1(a_1)=\tau_1(a_1,a_2)^{mn} \quad \textrm{and} \quad \theta_2(b_2)\theta_2(b_1b_2)^{-1}\theta_2(b_1)=\tau_2(b_1,b_2)^{mn}
\end{equation}
for all $a_1, a_2 \in A$ and $b_1, b_2 \in B$. Using \eqref{RRBC2}, we have  
 	\begin{equation}\label{rhodis1}
 		\rho(a_1,b)\rho(a_2, b)\rho(a_1 a_2, b)^{-1} =  \tau_1(a_1, a_2)  \tau_1(\beta_{b}(a_1), \beta_{b}(a_2))^{-1}
 	\end{equation}
for all $a_1,a_2\in A$ and $b\in B$. For  $a_1\in A$ and $b\in B$, taking product over all $a_2\in A$ gives
\begin{eqnarray*}
\prod_{a_2\in A}(\rho(a_1,b)\rho(a_2, b)\rho(a_1 a_2, b)^{-1}) &=& \prod_{a_2\in A}( \tau_1(a_1, a_2)  \tau_1(\beta_{b}(a_1), \beta_{b}(a_2))^{-1}).
\end{eqnarray*}
Since $\prod_{a_2\in A}\rho(a_1 a_2, b)^{-1}=\prod_{a_2\in A}\rho( a_2, b)^{-1}$, we get
\begin{eqnarray*}
\rho(a_1,b)^{m}  &=&  \prod_{a_2\in A}\tau_1(a_1, a_2)  \prod_{a_2\in A}\tau_1(\beta_{b}(a_1), \beta_{b}(a_2))^{-1}\\
&=&A(a_1) \prod_{a_2\in A}\tau_1(\beta_{b}(a_1), \beta_{b}(a_2))^{-1}\\
&=&A(a_1) \prod_{a_2\in A}\tau_1(\beta_{b}(a_1), a_2)^{-1},\quad\textrm{since}~\beta_b~\textrm{is an automorphism of A}\\
&=&A(a_1)(A(\beta_b(a_1)))^{-1}.
\end{eqnarray*}
Taking $n$-th powers on both the sides give
\begin{equation}\label{rho coboundary}
 \rho(a_1,b)^{mn} =A(a_1)^n(A(\beta_b(a_1)))^{-n}=\theta_1(a_1)(\theta_1(\beta_b(a_1)))^{-1}
\end{equation}
for all $a_1\in A$ and $b\in B$. 	
\par 

Similarly, using \eqref{RRBC3}, we have
\begin{equation}
\chi(a_2)\chi(a_1a_2)^{-1}\chi(a_1) = S^0(\rho(a_2, T(a_1))~\tau_1(a_1,\beta_{T(a_1)}(a_2)) )~ \tau_2(T(a_1), T(a_2) )^{-1}
\end{equation}
for all $a_1,a_2\in A$. Taking product over all $a_2\in A$ give
 	\begin{eqnarray*}
 		\prod_{a_2\in A}(\chi(a_2)\chi(a_1a_2)^{-1}\chi(a_1)) &=& \prod_{a_2\in A} \big(S^0(\rho(a_2, T(a_1))~\tau_1(a_1,\beta_{T(a_1)}(a_2)) ) ~\tau_2(T(a_1), T(a_2) )^{-1} \big).
\end{eqnarray*}
Since $\prod_{a_2\in A}\chi(a_1a_2)^{-1}=\prod_{a_2\in A}\chi(a_2)^{-1}$, we have
 	\begin{eqnarray*}
 		\chi(a_1)^{m}&=& \prod_{a_2\in A}S^0(\rho(a_2, T(a_1))) ~\prod_{a_2\in A} S^0(\tau_1(a_1,\beta_{T(a_1)}(a_2)) )~ \prod_{a_2\in A}\tau_2(T(a_1), T(a_2) )^{-1}\\
&=& S^0\big(\prod_{a_2\in A} \rho(a_2, T(a_1))\big)~\prod_{a_2\in A} S^0(\tau_1(a_1,a_2)) ~ \prod_{a_2\in A}\tau_2(T(a_1), T(a_2) )^{-1},\\ &&\textrm{since }\beta_{T(a_1)}\textrm{is an isomorphism of A.}\\
 		&=& S^0(A(a_1))~S^0 \big(\prod_{a_2\in A} \rho(a_2, T(a_1))\big)~\prod_{a_2\in A}\tau_2(T(a_1), T(a_2) )^{-1}.
 	\end{eqnarray*}
Taking $n$-th powers on both the sides give
 	\begin{eqnarray} \label{chi rel 1}
 		\chi(a_1)^{mn}&=& S^0(A(a_1))^n~S^0 \big(\prod_{a_2\in A} \rho(a_2, T(a_1))\big)^n~\prod_{a_2\in A}\tau_2(T(a_1), T(a_2) )^{-n}\\  \notag
 		&=& S^0(A(a_1)^n) ~S^0 \big(\prod_{a_2\in A} \rho(a_2, T(a_1))^n \big)~\prod_{a_2\in A}\tau_2(T(a_1), T(a_2) )^{-n},\\ \notag
 		&=& S^0(\theta_1(a_1))~S^0 \big(\prod_{a_2\in A}\rho(a_2, T(a_1))^{n}\big) ~\prod_{a_2\in A}\tau_2(T(a_1), T(a_2) )^{-n}. \notag
 	\end{eqnarray}
Now, using \eqref{RRBC1}, we obtain
 	\begin{equation}\label{rhodis2}
 		\rho(a, b_1 b_2)=\rho(\beta_{b_2}(a), b_1)\rho(a,b_2).
 	\end{equation}
for all $a \in A$ and $b_1, b_2 \in B$. Then, for $a_2\in A$ and $T(a_1)\in B$, \eqref{rhodis2} gives
 	\begin{eqnarray*}
 		\rho(a_2,T(a_1)^n)&=&\rho(\beta_{T(a_1)}(a_2), T(a)^{n-1})\rho(a_2,T(a_1))\\
 		&=&\rho(\beta_{T(a_1)}^2(a_2), T(a)^{n-2})\rho(\beta_{T(a_1)}(a_2),T(a_1))\rho(a_2,T(a_1))\\
 		&=&\rho(\beta_{T(a_1)}^{n-1}(a_2), T(a_1))\cdots\rho(\beta_{T(a_1)}(a_2),T(a_1))\rho(a_2,T(a_1))
 	\end{eqnarray*}
Taking product over all $a_2\in A$ yield
 	\begin{eqnarray*}
 		\prod_{a_2\in A}\rho(a_2,T(a_1)^n)&=&\prod_{a_2\in A} \big(\rho(\beta_{T(a_1)}^{n-1}(a_2), T(a_1))\dots\rho(\beta_{T(a_1)}(a_2),T(a_1))\rho(a_2,T(a_1))\big)\\
 		&=&\prod_{a_2\in A}\rho(a_2,T(a_1))^n,\quad\textrm{since}~\beta_{T(a_1)}^i\textrm{ is an automorphism of A for each $i$.}
 	\end{eqnarray*}
 	However, $\rho(a_2,T(a_1)^n)=1$ since $T(a_1)\in B$, and hence we obtain 
\begin{equation} \label{chi rel 2}
\prod_{a_2\in A}\rho(a_2,T(a_1))^n=1.
\end{equation}
Using \eqref{Mcond}, we have
\begin{eqnarray}
\prod_{a_2\in A}\tau_2(T(a_1), T(a_2) )^{-n}&=&\prod_{a_2\in A} \big(B(T(a_1)T(a_2))~B(T(a_1))^{-1}~B(T(a_2))^{-1} \big)\\  \label{chi rel 3} \notag
&=&B(T(a_1))^{-m}~\prod_{a_2\in A} \big(B(T(a_1)T(a_2))~B(T(a_2))^{-1}\big )\\ \notag
&=&B(T(a_1))^{-m}~\prod_{a_2\in A} \big(B(T(a_1\beta_{T(a_1)}(a_2)))B(T(a_2))^{-1}\big)\\ \notag
&=&B(T(a_1))^{-m},\quad\textrm{since}~\beta_{T(a_1)}\textrm{ is an automorphism of A}\\ \notag
&=&\theta_2(T(a_1))^{-1}. \notag
\end{eqnarray}
Using  \eqref{chi rel 1},  \eqref{chi rel 2} and  \eqref{chi rel 3}, we obtain
\begin{equation}\label{chi coboundary}
\chi(a_1)^{mn}=S^0(\theta_1(a_1))~\theta_2(T(a_1))^{-1}
\end{equation}
for all $a_1\in A$. The theorem now follows due to \eqref{theta1 theta2 coboundary}, \eqref{rho coboundary} and \eqref{chi coboundary}.
 \end{proof}
 For each integer $n>1$, let us set $G_n=\{z\in \mathbb{C} \mid z^n=1\}$, the group of $n$-th roots of unity.
 
 \begin{thm}\label{VI2}
Let $\mathcal{A}=(A,B,\beta,T)$ be a relative Rota-Baxter group. Then every element of $M_{RRB}(\mathcal{A})$ of order dividing $n$  has a representative $(\tau_1,\tau_2,\rho,\chi)$ such that $\IM(\tau_1),\IM(\tau_2),\IM(\rho),\IM(\chi)\subseteq G_{n}$. Moreover, if $\mathcal{A}$ is finite, then $M_{RRB}(\mathcal{A})$ is finite.
 \end{thm}

 \begin{proof}
Let $(\tau_1,\tau_2,\rho,\chi)$ be a 2-cocycle representing an element of $M_{RRB}(\mathcal{A})$ of order dividing $n$. Then there exist maps $k_1:A\to \mathbb{C}^{\times}$ and $k_2:B\to \mathbb{C}^{\times}$ with $k_1(1)=1$, $k_2(1)=1$ and $(\tau_1^n,\tau_2^n,\rho^n,\chi^n)=(\partial^1(k_1),\partial^1(k_2),\lambda_1,\lambda_2)$, where $\lambda_1, \lambda_2$ are given by
 	\begin{eqnarray*}
 		\lambda_1(a_1,b_1)&=&k_1(a_1)k_1(\beta_{b_1}(a_1))^{-1},\\
 		\lambda_2(a_1)&=&S^0(k_1(a_1))k_2(T(a_1))^{-1}
 	\end{eqnarray*}
 for all $a_1\in A$ and $b_1\in B$.
\par 
We define maps $\theta_1:A \to \mathbb{C}^\times$ and $\theta_2:B \to \mathbb{C}^\times$ by first setting $\theta_1(1)=1$ and $\theta_2(1)=1$. For $a\in A\setminus\{1\}$ and $b\in B\setminus\{1\}$, let $\theta_1(a)$ be the $n$-th root of  $k_1(a)^{-1}$ and $\theta_2(b)$ be the $n$-th root of $k_2(b)^{-1}$. Define the maps $\tau'_1,\tau'_2,\rho',\chi'$ such that 
 	\begin{eqnarray*}
 		\tau'_1(a_1,a_2)&=&\tau_1(a_1,a_2) \theta_1(a_2)\theta_1(a_1a_2)^{-1} \theta_1(a_1),\\
 		\tau'_2(b_1,b_2)&=&\tau_2(b_1,b_2) \theta_2(b_2)\theta_2(b_1b_2)^{-1}\theta_2(b_1),\\
 		\rho'_1(a_1,b_1)&=&\rho_1(a_1,b_1)\theta_1(a_1)\theta_1(\beta_{b_1}(a_1))^{-1},\\
 		\chi'(a_1)&=&\chi(a_1)~S^0(\theta_1(a_1))\theta_2(T(a_1))^{-1}
 	\end{eqnarray*}
for $a_1, a_2 \in A$ and $b_1, b_2 \in B$.	It can be checked that $(\tau'_1,\tau'_2,\rho',\chi')$ is a 2-cocycle. This shows that $(\tau'_1,\tau'_2,\rho',\chi')$ and $(\tau_1,\tau_2,\rho,\chi)$ are cohomologous 2-cocycles. Moreover, we have
 	\begin{equation*}
 		\tau'_1(a_1,a_2)^{n}=\tau'_2(b_1,b_2)^{n}=\rho'(a_1,b_1)^{n}=\chi'(a_1)^n=1,
 	\end{equation*}
which establishes the first assertion.
\par
Suppose that $\mathcal{A}$ is finite with $|A|=m$ and $|B|=n$. By Theorem \ref{VI1}, every element of $M_{RRB}(\mathcal{A})$ has order dividing $mn$. In view of the first assertion,  every element of $M_{RRB}(\mathcal{A})$ of order dividing $mn$ is determined by a map $A\times A\times B\times B\longrightarrow G_{mn}\times G_{mn}\times G_{mn}\times G_{mn}$. Since there are only a finite number of such maps, it follows that $M_{RRB}(\mathcal{A})$ is finite.
 \end{proof}
 
In what follows, $G'$ denote the commutator subgroup of a group $G$. We now introduce the commutator subgroup of a relative Rota-Baxter group \cite[Definition 6.5]{NM1}.
 
 \begin{defn}\label{Commutator}
The commutator subgroup of a relative Rota-Baxter group $(H, G, \phi, R)$ is defined to be the relative Rota-Baxter group $(H^{\phi}, G, \phi|, R|)$, where $H^{\phi}$ is the subgroup of $H$ generated by its commutator subgroup $H'$ and the subgroup $H^{(2)}$, which itself is generated by the set  $\{\phi_{g}(h)h^{-1} \mid  h \in H  \mbox{ and }  g \in G \}$.  We denote the commutator subgroup by $(H, G, \phi, R)^{\prime}$.
 \end{defn}

Recall that, the commutator $H^\prime$ of a skew left brace $(H, \cdot, \circ)$ is defined as the subgroup of $H^{(\cdot)}$ generated by the commutator subgroup of $H^{(\cdot)}$ and the subgroup generated by the set $\{a^{-1} \cdot (a \circ b) \cdot b^{-1} \mid a, b \in H\}$. Below we examine \cite[Lemma 3.2]{TV23} in the context of relative Rota-Baxter groups.

 \begin{lemma}\label{SC1}
 	Let $\mathcal{A}=(A,B,\beta,T)$ be a relative Rota-Baxter group and $\mathcal{K}=(K,L, \alpha,S )$ a relative Rota-Baxter subgroup of $\mathcal{A}$ such that $S$ is surjective. If $\mathcal{K}\subseteq \mathcal{A}'$, then every homomorphism $\mathcal{A}\to \mathcal{C}$ restricts to the trivial homomorphism $\mathcal{K} \to \mathcal{C}$ .
 \end{lemma}

 \begin{proof}
Let $(\psi,\eta): (A,B,\beta,T) \to (\mathbb{C}^{\times},\mathbb{C}^{\times},\alpha^0,S^0)$ be a homomorphism of relative Rota-Baxter groups.
Obviously, for elements $hgh^{-1}g^{-1} \in A'$, we have $\psi(hgh^{-1}g^{-1})=1$. Since $\alpha^0$ is trivial, for $a \in A$, we have
$$ \psi \beta_b(a) =\alpha^0_{\eta(b)}\psi(a)= \psi(a),$$
and hence  $\psi(\beta_b(a)a^{-1})=\psi(a)\psi(a^{-1})=1$. This shows that the map $\psi$ restricted to $K$ is trivial. Further, since $\eta S(k)=\eta T(k)=S^0\psi(k)=1$ for all $k \in K$ and  $S$ is surjective, it follows that $\eta(l)=1$ for all $l\in L$. This proves the assertion.
\end{proof}

\begin{remark}
The converse of Lemma \ref{SC1} may not hold in general. For instance, let $\mathcal{A}=(A_5\times \mathbb{Z}_2, A_5,\beta, T)$, where $\beta$ is trivial and $T(a,b)= a$. Since $A_5$ is simple and non-abelian, every group homomorphism $\eta: A_5 \to \mathbb{C}^\times$ is trivial. It follows from \eqref{rbb datum morphism} that for any homomorphism $(\psi, \eta):\mathcal{A} \to \mathcal{C}$ of relative Rota-Baxter groups, the component group homomorphisms $\psi$ and $\eta$ are trivial.  Taking $\mathcal{K}=\mathcal{A}$, we see that every homomorphism  $\mathcal{A} \to \mathcal{C}$ restricts to a trivial homomorphism $\mathcal{K} \to \mathcal{C}$. However, $\mathcal{K}$ is not contained in $\mathcal{A}'$ since $\mathcal{A}'=(A_5\times \{1\}, A_5, \beta |,T|)$.
\end{remark}

However, under some condition, the converse of Lemma \ref{SC1} does holds.

\begin{lemma}\label{converse SC1}
Let $\mathcal{A}=(A,B,\beta,T)$ be a finite bijective relative Rota-Baxter group and $\mathcal{K}=(K,L, \alpha,S )$ a relative Rota-Baxter subgroup of $\mathcal{A}$ such that $S$ is surjective. Then every homomorphism $\mathcal{A}\to \mathcal{C}$ restricts to the trivial homomorphism $\mathcal{K} \to \mathcal{C}$ if and only if $\mathcal{K}\subseteq \mathcal{A}'$.
\end{lemma}

\begin{proof}
If $\mathcal{K}\subseteq \mathcal{A}'$, then by Lemma \ref{SC1} every homomorphism $\mathcal{A}\rightarrow\mathcal{C}$ restricts to the trivial homomorphism $\mathcal{K} \rightarrow\mathcal{C}$. Conversely, suppose that there exists $a\in K\cap (A\setminus A^{\beta})$, where $A^{\beta}$ is as defined in Definition \ref{Commutator}. Let $\overline{a}$ denote the coset of $a$ in $A/A^{\beta}$ and let $|\overline{a}|=k$, where $k \ge 2$. Let $\xi$ be a primitive $k$-th root of unity. Since $A/A^{\beta}$ is a finite abelian group, the map $A/A^{\beta} \to\mathbb{C}^{\times}$ that map $\overline{a} \mapsto \xi$ and $\bar{x} \to 1$ for all $\bar{x} \not\in \langle \bar{a} \rangle$, is a group homomorphism. The composition of this map with the natural surjection $A\rightarrow A/A^{\beta}$ yields a group 
homomorphism $\psi:A\to \mathbb{C}^{\times}$. Since $T:A \to B$ is a bijection, we can define a map $\eta:B\to\mathbb{C}^{\times}$ such that $\eta(T(a))=S^0(\psi(a))$ for all $a \in A$. Also, since $\beta_b(a)a^{-1}\in A^{\beta}$, it follows from the construction of $\psi$ that $\psi(\beta_b(a)a^{-1})=1$, and hence $\psi \beta_b(a)= \psi(a) =\alpha^0_{\eta(b)}\psi(a)$ for all $a \in A$. The identity $\psi \beta_b= \alpha^0_{\eta(b)}\psi$ implies that $\eta$ is actually a group homomorphism, and hence $(\psi,\eta): \mathcal{A} \to \mathcal{C}$ is a homomorphism of relative Rota Baxter groups such that $\psi(a)=\xi \neq1$. This proves the lemma.
\end{proof}

 \begin{thm}\label{existence of Transgression Isomorphism}
Every finite relative Rota-Baxter group admits at least one central extension such that its corresponding transgression map is an isomorphism.
 \end{thm}
 \begin{proof}
 	Let $\mathcal{A}=(A,B,\beta,T)$ be a finite relative Rota-Baxter group. By Theorem \ref{VI2}, $M_{RRB}(\mathcal{A})$ is a finite abelian group, and hence we can write 
$$ M_{RRB}(\mathcal{A})= \langle c_1\rangle \oplus \cdots \oplus \langle c_m\rangle,$$
an internal direct product of cyclic groups $\langle c_i\rangle$ of order $d_i$.  Again, by Theorem \ref{VI2}, each $c_i$ has a representative $(\tau_{1_i},\tau_{2_i},\rho_i,\chi_i) \in \Z_{RRB}^2(\mathcal{A}, \mathcal{C})$, where $\tau_{1_i}: A\times A\to \mathbb{C}^{\times}$, $\tau_{2_i}: B\times B\to \mathbb{C}^{\times}$, $\rho_i: A\times B\to \mathbb{C}^{\times}$ and $\chi_i: A\to \mathbb{C}^{\times}$  are maps  such that $\IM(\tau_{1_i}),\IM(\tau_{2_i}),\IM(\rho_i),\IM(\chi_i)\subseteq G_{d_i}$ for each $i$. Let $G=G_{d_1}\times\cdots\times G_{d_m}$. Let $\mathcal{G}_i=(G_{d_i},G_{d_i},\alpha^0|,S^0|)$ denote the relative Rota-Baxter subgroup of $\mathcal{C}$  and $\mathcal{G}=(G,\,G, \,\alpha^0|\times \cdots \times\alpha^0|, \, S^0|\times\cdots \times S^0|)$ denote their product. Since $M_{RRB}(\mathcal{A})\cong G$,  we have homomorphisms 
$$ r_i: \Ho^2_{RRB}\left(\mathcal{A},\mathcal{G}_i\right) \longrightarrow G$$
induced by the inclusion $G_{d_i}\hookrightarrow \mathbb{C}^{\times}$. For each $i$, let $\kappa_i\in \Ho^2_{RRB}(\mathcal{A},\mathcal{G}_i)$ be 
 	the equivalence class of $(\tau_{1_i},\tau_{2_i},\rho_i,\chi_i)$ considered as an element of $\Z^2_{RRB}(\mathcal{A},\mathcal{G}_i)$. Then 
 	$c_i=r_i(\kappa_i)$.  Let $\kappa\in \Ho^2_{RRB}(\mathcal{A},\mathcal{G})$ be the image of $\left(\kappa_1,\dots , \kappa_m\right)$ under the isomorphism 
$$ \prod_{i=1}^m \Ho^2_{RRB}\left(\mathcal{A},\mathcal{G}_i\right) \cong \Ho^2_{RRB}\left(\mathcal{A},\mathcal{G}\right)$$
given by Proposition \ref{cohomology with product coeff}. Suppose that 
$$
\mathcal{E}: \quad {\bf 1} \to \mathcal{G}=(G,\,G, \,\alpha^0|\times \cdots \times\alpha^0|, \,S^0|\times\cdots \times S^0|) \stackrel{(i_1, i_2)}{\longrightarrow} \mathcal{H}= (H,G, \phi, R) \stackrel{(\pi_1, \pi_2)}{\longrightarrow} \mathcal{A}=(A,B, \beta, T) \to {\bf 1}
$$
 is a central extension associated with $\kappa$. We claim that $\mathcal{E}$ is the desired central extension. Let 	
 $$ \Tra : {\Hom_{RRB}{(\mathcal{G},\mathcal{C})}}{\longrightarrow}M_{RRB}(\mathcal{A})$$
be the transgression homomorphism associated with $\kappa$ and $p_i: G\to \mathbb{C}^{\times}$ be the projection given by $p_i(z_1,\dots, z_m)=z_i$. Then, $\Tra\left((p_i,p_i)\right) = c_i$, and hence $\Tra$ is surjective. Since both $\mathcal{G}$ and $\mathcal{C}$ are bijective relative Rota-Baxter groups, using \cite[Proposition 4.5]{NM1} and the well-known fact that $\Hom_{Group}(G, \mathbb{C}^{\times})\cong G$, we obtain
$$ \Hom_{RRB}\left(\mathcal{G},\mathcal{C}\right)\cong \Hom_{Group}(G, \mathbb{C}^{\times})\cong G \cong M_{RRB}\left(\mathcal{A}\right).
$$
Since $M_{RRB}(\mathcal{A})$ is finite, it follows that $\Tra$ is an isomorphism. 
 \end{proof}
 \medskip

 \section{Relationship between central extensions and Schur covers}\label{isoclinic Schur covers}
In this section, we investigate the connection between two central extensions constructed in Theorem \ref{existence of Transgression Isomorphism} and introduce the concept of a Schur cover for a relative Rota-Baxter group.
\par

Let $(H, G, \phi, R)$ be a relative Rota-Baxter group. Then there are maps $$\omega_H, \omega^{\phi}_H: \big( H/ \Z^{\phi}_{R}(H)\big) \times \big(H/ \Z^{\phi}_{R}(H)\big)  \rightarrow H^\phi$$ defined by
	$$\omega_{H}(\overline{h}_1, \overline{h}_2)=h_1 h_2 h_1^{-1} h_2^{-1} \quad \textrm{and} \quad \omega^{\phi}_{H}(\overline{h}_1, \overline{h}_2)= \phi_{R(h_1)}(h_2) h_2^{-1}$$
for all $h_1, h_2 \in H$.

For a relative Rota-Baxter group $(H, G, \phi, R)$, consider the relative Rota-Baxter group $\I (H, G, \phi, R) := (H, R(H), \phi|, R|)$. It is evident that the skew left braces induced by $(H, G, \phi, R)$ and $\I (H, G, \phi, R)$ are the same.

\begin{defn}
	Two relative Rota-Baxter groups  $(H, G, \phi, R)$ and $(K, L, \varphi, S)$ are isoclinic  if there are isomorphisms of relative Rota-Baxter groups 
	$$(\psi_1, \eta_1): (H, G, \phi, R) / \Z(H, G,  \phi, R ) \longrightarrow (K, L, \varphi, S) / \Z(K, L, \varphi, S)$$
	and 
	$$(\psi_2, \eta_2): (H, G, \phi, R)^\prime \longrightarrow (K, L, \varphi, S)^\prime$$
	such that the following diagram commutes
	\begin{align}\label{cd}
		\begin{CD}
			H^{\phi}  @<{\omega_H}<<(H/ \Z^{\phi}_{R}(H)) \times (H/ \Z^{\phi}_{R}(H))@>{\omega^{\phi}_H}>> H^{\phi} \\
			@V{\psi_2}VV @V{\psi_1 \times \psi_1}VV @V{\psi_2}VV\\
			K^{\varphi} @<{\omega_K}<< (K/ \Z^{\varphi}_{S}(K)) \times (K/ \Z^{\varphi}_{S}(K))@>{\omega^{\varphi}_K}>> K^{\varphi}.
		\end{CD}
	\end{align}
\end{defn}

\begin{defn}
The relative Rota-Baxter group $(H, G, \phi, R)$ is  weakly isoclinic to the relative Rota-Baxter group $(K, L, \varphi, S)$  if there is an isomorphism 
	$$(\psi_1, \eta_1): (H, G, \phi, R) / \Z(H, G,  \phi, R ) \longrightarrow (K, L, \varphi, S) / \Z(K, L, \varphi, S)$$
	and  a homomorphism
	$$(\psi_2, \eta_2): \I((H, G, \phi, R)^\prime) \longrightarrow \I ((K, L, \varphi, S)^\prime)$$
	 of relative Rota-Baxter groups 	such that $\psi_2$ is an isomorphism of groups and  the following diagram commutes
	\begin{align}\label{cd1}
		\begin{CD}
			H^{\phi}  @<{\omega_H}<<(H/ \Z^{\phi}_{R}(H)) \times (H/ \Z^{\phi}_{R}(H))@>{\omega^{\phi}_H}>> H^{\phi} \\
			@V{\psi_2}VV @V{\psi_1 \times \psi_1}VV @V{\psi_2}VV\\
			K^{\varphi} @<{\omega_K}<< (K/ \Z^{\varphi}_{S}(K)) \times (K/ \Z^{\varphi}_{S}(K))@>{\omega^{\varphi}_K}>> K^{\varphi}.
		\end{CD}
	\end{align}
\end{defn}

We note that unlike isoclinism, weak isoclinism is not an equivalence relation on the class of relative Rota-Baxter groups. We recall the following result from \cite[Proposition 6.10]{NM1}.

\begin{prop}\label{indiso}
Let $(H,G,\phi,R)$ and $(K,L,\varphi,S)$ be isoclinic relative Rota-Baxter groups. Then the following assertions hold:
\begin{enumerate}
\item $\im(R)/ (\im(R) \cap \Ker(\phi)) \cong \im(S)/( \im(S) \cap \Ker(\varphi))$.
\item $\langle \phi_{R(h_1)}(h_2)h_2^{-1} \mid h_1, h_2 \in H \rangle \cong \langle \varphi_{S(k_1)}(k_2)k_2^{-1} \mid k_1, k_2 \in K \rangle$.
\item  $H/ \Z^{\phi|}_{R|}(H) \cong K/ \Z^{\varphi|}_{S|}(K)$.
\item $\I(H,G,\phi,R)$ and $\I(K,L,\varphi,S)$ are isoclinic relative Rota-Baxter groups.
\end{enumerate}
\end{prop}

Next, we observe a generalization of the preceding result to weak isoclinism of relative Rota-Baxter groups.

\begin{prop}\label{wiso1}
If $(H,G,\phi,R)$ is  weakly isoclinic to  $(K,L,\varphi,S)$, then  $\I(H,G,\phi,R)$ is  weakly isoclinic to  $\I(K,L,\varphi,S).$
\end{prop}

\begin{proof}
Suppose that	$(H, G, \phi, R)$ is  weakly isoclinic to  $(K, L, \varphi, S)$. Let
$$(\psi_1, \eta_1): (H, G, \phi, R) / \Z(H, G,  \phi, R ) \longrightarrow (K, L, \varphi, S) / \Z(K, L, \varphi, S)$$
and
$$(\psi_2, \eta_2): \I((H, G, \phi, R)^\prime) \longrightarrow \I ((K, L, \varphi, S)^\prime)$$
be the desired homomorphisms satisfying the condition of weak isoclinism. As weak isoclinism partially fulfills the conditions of isoclinism, it ensues that the assertions $(1)$, $(2)$ and $(3)$ of Proposition \ref{indiso}  hold for weakly isoclinic relative Rota-Baxter groups. It follows from the proof of Proposition \ref{indiso}(1)  that $\eta_1$ maps $\im(R)/(\im(R) \cap \Ker(\phi))$ onto $\im(S)/(\im(S) \cap \Ker(\varphi))$. Similarly, proof of Proposition \ref{indiso}(3)  shows that $\psi_1$ induces an isomorphism $H/ \Z^{\phi|}_{R|}(H)$ onto $K/ \Z^{\varphi|}_{S|}(K)$. Thus,
$$(\bar{\psi}_1, \bar{\eta}_1|): \I(H, G, \phi, R) / \Z(\I(H, G,  \phi, R )) \longrightarrow \I(K, L, \varphi, S) / \Z(\I(K, L, \varphi, S))$$
is an isomorphism of relative Rota-Baxter groups, where $(\bar{\psi}_1, \bar{\eta}_1)$ is induced from $(\psi_1, \eta_1)$. Note that $ \I((H, G, \phi, R)^\prime)=(H^{\phi}, R(H^{\phi}),\phi|, R| )$ and $\I((K, L, \varphi, S)^\prime)=(K^{\varphi}, S(K^{\varphi}),\varphi|, S| )$. Assertion Proposition \ref{indiso}(2) gives a homomorphism $$(\psi_2|, \eta_2|): (H^{\phi|}, R(H^{\phi|}), \phi|, R|) \longrightarrow (K^{\varphi|}, S(K^{\varphi|}), \varphi|, S|),$$
where $(\psi_2|, \eta_2|)$ is the restriction of $(\psi_2, \eta_2)$ and $\psi_2|$ is an  isomorphism of groups. Note that $\I( (\I(H, G, \phi, R))^\prime)=(H^{\phi|}, R(H^{\phi|}), \phi|, R|)$ and $\I( (\I(K, L, \varphi, S))^\prime)=(K^{\varphi|}, S(K^{\varphi|}), \varphi|, S|)$. Hence, $\I(H,G,\phi,R)$ is  weakly isoclinic to  $\I(K,L,\varphi,S)$ due to commutativity of the diagram \eqref{cd1}.
\end{proof}

Isoclinism of skew left braces has been introduced recently in \cite{TV23}. Given a skew left brace $(H, \cdot, \circ)$, there are maps  $\theta_H, \theta^{*}_H: (H/ \Ann(H)) \times (H/ \Ann(H)) \longrightarrow H^\prime$ given by $$\theta_H(\overline{a}, \overline{b})= a \cdot b \cdot a^{-1} \cdot b^{-1}$$ and $$\theta^{*}_H(\overline{a}, \overline{b})= a^{-1} \cdot (a \circ b) \cdot a^{-1}.$$ Two skew left braces $(H, \cdot, \circ)$ and $(K, \cdot, \circ)$ are said to be isoclinic if there are isomorphisms $\xi_1 : H/ \Ann(H)\longrightarrow K/ \Ann(K)$ and $\xi_2: H^\prime \longrightarrow K^\prime$ such that the following diagram commutes
\begin{align}\label{cdsb}
	\begin{CD}
		H^\prime  @<{\theta_H}<<(H/ \Ann(H)) \times (H/ \Ann(H))@>{\theta^{*}_H}>> H^\prime \\
		@V{\xi_2}VV @V{\xi_1 \times \xi_1}VV @V{\xi_2}VV\\
		K^\prime @<{\theta_K}<< (K/ \Ann(K)) \times (K/ \Ann(K))@>{\theta^{*}_K}>> K^\prime.
	\end{CD}	
\end{align}

The next result extends  \cite[Theorem 6.11]{NM1} to weak isoclinism of relative Rota-Baxter groups.

\begin{thm}\label{almost isoclinism rrbg implies isoclinism slb}
If $(H,G,\phi,R)$ is weakly isoclinic to  $(K,L,\varphi,S)$, then their induced skew left braces $H_R$ and $K_S$ are isoclinic.
\end{thm}

\begin{proof}
If $(H,G,\phi,R)$ is weakly isoclinic to  $(K,L,\varphi,S)$, then by Proposition \ref{wiso1}, $\I(H,G,\phi,R)$ is weakly isoclinic to  $\I(K,L,\varphi,S)$. Thus, there exists an isomorphism 
	$$(\psi_1, \eta_1): \I(H, G, \phi, R) / \Z(\I(H, G,  \phi, R )) \longrightarrow \I(K, L, \varphi, S) / \Z(\I(K, L, \varphi, S))$$
	and  a homomorphism
	$$(\psi_2, \eta_2): \I((\I(H, G, \phi, R))^\prime) \longrightarrow \I ((\I(K, L, \varphi, S))^\prime)$$
	such that $\psi_2$ is an isomorphism of groups and  the following diagram commutes
	\begin{align}
		\begin{CD}\label{cd3}
			H^{\phi|}  @<{\omega_H}<<\big( H/ \Z^{\phi|}_{R|}(H)\big) \times \big( H/ \Z^{\phi|}_{R|} (H)\big)@>{\omega^{\phi|}_H}>> H^{\phi|} \\
			@V{\psi_2}VV @V{\psi_1 \times \psi_1}VV @V{\psi_2}VV\\
			K^{\varphi|} @<{\omega_K}<< \big(K/ \Z^{\varphi|}_{S|} (K)\big) \times \big(K/ \Z^{\varphi|}_{S|}(K) \big)@>{\omega^{\varphi|}_K}>> K^{\varphi|}.
		\end{CD}
	\end{align}
It follows from \cite[Proposition 6.2(3)]{NM1} that the skew left brace induced by $\I(H, G, \phi, R) / \Z(\I(H, G, \phi, R))$ is identical to $H_R/ \Ann(H_R)$. Notably, both $(H, G, \phi, R)$ and $\I(H, G, \phi, R)$ give rise to identical skew left braces. Furthermore, the skew left brace induced by $(\I(H, G, \phi, R))^\prime$ is identical to $H^\prime_R$. Thus, the skew left brace induced by $\I((\I(H, G, \phi, R))^\prime)$ is also identical to $H^\prime_R.$ Since every morphism of relative Rota Baxter groups induce a morphism of corresponding skew left braces, by utilizing \eqref{cd3}, we obtain the following commutative diagram
	$$
	\begin{CD}
		H_R^\prime  @<{\omega_H}<< \big(H_R/\Ann(H_R) \big)\times \big(H_R/\Ann(H_R) \big)@>{\omega^{\phi|}_H}>> H_R^\prime\\
		@V{\psi_2}VV @V{\psi_1 \times \psi_1}VV @V{\psi_2}VV\\
		K_S^\prime @<{\omega_K}<< \big(K_S/\Ann(K_S) \big) \times  \big(K_S/\Ann(K_S) \big)  @>{\omega^{\varphi|}_K}>> K_S^\prime.
	\end{CD}
	$$
	This shows that $H_R$ and $K_S$ are isoclinic.
\end{proof}

\begin{lemma}\label{Centrecondn}
Let $\mathcal{A}=(A,B,\beta,T)$ be a relative Rota-Baxter group and $\mathcal{K}=(K,L, \alpha,S )$ a trivial  $\mathcal{A}$-module. Let $(\tau_1,\tau_2,\rho,\chi)\in \Z^2_{RRB}(\mathcal{A},\mathcal{K})$ and $\mathcal{H}=(H,G,\phi,R)=(A\times_{\tau_1}K,B\times_{\tau_2}L,\phi,R)$ the corresponding extension. Then
\begin{eqnarray*}
\Z^{\phi}_{R}(H) &=&\{(a,k)\in H ~\mid~ a\in \Z(A),~ \tau_1(a,x)=\tau_1(x,a), ~\beta_{T(a)}(x)=x,\\
&&\rho(x,T(a))=1,~ \beta_b(a)=a, ~\rho(a,b)=1 ~ \textrm{for all}~ x \in A ~ \textrm{and}~ b\in B\},\\
\Ker(\phi) &=&\{(b,l)\in G ~\mid ~\beta_b(a)=a, ~\rho(a,b)=1~ \textrm{for all}~ a\in A\}.
\end{eqnarray*}
\end{lemma}
\begin{proof}
Through direct calculations, we see that
\begin{eqnarray*}
\Z(H) &=& \{ (a,k) \in H~ \mid~  a \in \Z(A) \mbox{ and } \tau_1(a,x)=\tau_1(x,a) \mbox{ for all } x \in A  \},\\
\Fix(\phi) &=& \{ (a,k) \in H ~\mid~  \beta_b(a)=a  \mbox{ and } \rho(a,b)=1 \mbox{ for all } b\in B\},\\
\ker (\phi\; R) &=& \{ (a,k) \in H ~\mid~  \beta_{T(a)}(x)=x \mbox{ and } \rho(x,T(a))=1 \mbox{ for all } x \in A  \},\\ 
\Ker(\phi) &=& \{(b,l)\in G ~\mid ~\beta_b(a)=a, ~\rho(a,b)=1~ \textrm{for all}~ a\in A\}.
\end{eqnarray*}
The remaining assertion follows from the definition of $\Z^{\phi}_{R}(H)$.
\end{proof}

\begin{remark}\label{Generators}
The group $H^{\phi}$ is generated by the elements 
	$$ \big(a_1a_2a_1^{-1}a_2^{-1},\tau_1(a_1,a_2)\tau_1(a_1,a_1^{-1})^{-1}\tau_1(a_1a_2,a_1^{-1})\tau_1(a_2,a_2^{-1})^{-1}\tau_1(a_1a_2a_1^{-1},a_2^{-1}) \big),$$
	$$\big(\beta_b(a)a^{-1},\rho(a,b)\tau_1(a,a^{-1})^{-1}\tau_1(\beta_b(a),a^{-1}) \big)$$ and the generators of $G'$ are exactly the elements
	$$\big(b_1b_2b_1^{-1}b_2^{-1},\tau_2(b_1,b_2)\tau_2(b_1,b_1^{-1})^{-1}\tau_2(b_1b_2,b_1^{-1})\tau_2(b_2,b_2^{-1})^{-1}\tau_2(b_1b_2b_1^{-1},b_2^{-1})\big),$$ 
	where $a,a_1,a_2\in A$ and $b,b_1,b_2\in B$.
\end{remark}

\begin{thm}\label{thm:almost isoclinism}
Let $\mathcal{A} = (A, B, \beta, T)$ be a finite relative Rota-Baxter group. Let $$M_{RRB}(\mathcal{A})\cong \langle c_1\rangle\times\cdots\times\langle c_m\rangle$$ and $$\mathcal{K} =(K, L, \alpha, S) = \big( M_{RRB}(\mathcal{A}), \,M_{RRB}(\mathcal{A}), \, \alpha^0|\times\cdots \times \alpha^0|, \, S^0|\times\dots\times S^0| \big),$$ 
 where $ c_i$ is the complex $n_i$-th root of unity for each $i$ and $\mathcal{K}$ is the trivial relative Rota Baxter group viewed as a trivial  $\mathcal{A}$-module. If 
	\begin{equation*}\label{V1}
		\mathcal{E}_1 : \quad  {\bf 1} \longrightarrow \mathcal{K} \stackrel{(i_1, i_2)}{\longrightarrow} \mathcal{H}_1= (H_1,G_1, \phi_1, R_1) \stackrel{(\pi_1, \pi_2)}{\longrightarrow} \mathcal{A}\longrightarrow {\bf 1}
	\end{equation*} and 
	\begin{equation*}\label{V1}
		\mathcal{E}_2 : \quad  {\bf 1} \longrightarrow \mathcal{K} \stackrel{(i_1, i_2)}{\longrightarrow}\mathcal{H}_2=  (H_2,G_2, \phi_2, R_2) \stackrel{(\pi_1, \pi_2)}{\longrightarrow} \mathcal{A}\longrightarrow {\bf 1}
	\end{equation*}
are two central extensions  of  $\mathcal{A}$ by $\mathcal{K}$ such that their corresponding transgression maps $ \Tra_{\mathcal{H}_1}:\Hom_{RRB}(\mathcal{K},\mathcal{C})\longrightarrow M_{RRB}(\mathcal{A})$ and $\Tra_{\mathcal{H}_2}:\Hom_{RRB}(\mathcal{K},\mathcal{C})\longrightarrow M_{RRB}(\mathcal{A})$ are isomorphisms, then	$\mathcal{H}_1$ and $\mathcal{H}_2$ are weakly isoclinic.
\end{thm}

\begin{proof}
Let $(\tau^1_1,\tau^1_2,\rho^1,\chi^1), (\tau^2_1,\tau^2_2,\rho^2,\chi^2)  \in \Z^2_{RRB}(\mathcal{A}, \mathcal{K})$ be two 2-cocycles corresponding to the extensions $\mathcal{H}_1$ and $\mathcal{H}_2$, respectively. By Proposition \ref{cohomology with product coeff}, there exist $(\tau^1_{1_i},\tau^1_{2_i},\rho^1_i,\chi^1_i), (\tau^2_{1_i},\tau^2_{2_i},\rho^2_i,\chi^2_i) \in \Z^2_{RRB}(\mathcal{A}, \langle c_i \rangle)$ for each $i$, such that
	\begin{eqnarray*}
		\tau^1_1(a_1,a_2)&=& \big(\tau^1_{1_1}(a_1,a_2),\dots,\tau^1_{1_m}(a_1,a_2) \big),\\
		\tau^1_2(b_1,b_2)&=& \big(\tau^1_{2_1}(b_1,b_2),\dots,\tau^1_{2_m}(b_1,b_2) \big),\\
		\rho^1(a_1,b_1)&=& \big(\rho^1_1(a_1,b_1),\dots,\rho^1_m(a_1,b_1) \big),\\
		\chi^1(a_1)&=& \big(\chi_1^1(a_1),\dots,\chi^1_m(a_1) \big),\\
\tau^2_1(a_1,a_2)&=& \big(\tau^2_{1_1}(a_1,a_2),\dots,\tau^2_{1_m}(a_1,a_2) \big),\\
		\tau^2_2(b_1,b_2)&=& \big(\tau^2_{2_1}(b_1,b_2),\dots,\tau^2_{2_m}(b_1,b_2) \big),\\
		\rho^2(a_1,b_1)&=& \big(\rho^2_1(a_1,b_1),\dots,\rho^2_m(a_1,b_1) \big),\\
		\chi^2(a_1)&=& \big(\chi^2_1(a_1),\dots,\chi^2_m(a_1)\big)
	\end{eqnarray*}
for all $a_1, a_2 \in A$ and $b_1,b_2 \in B$. We can view each $ (\tau^1_{1_i},\tau^1_{2_i},\rho^1_i,\chi^1_i)$ as a representative of some element of $M_{RRB}(\mathcal{A})$. Similarly, each $(\tau^2_{1_i},\tau^2_{2_i},\rho^2_i,\chi^2_i)$ represents some element of $M_{RRB}(\mathcal{A})$. We are given that the transgression maps
	$$\Tra_{\mathcal{H}_1}, \Tra_{\mathcal{H}_2}: \Hom_{RRB}(\mathcal{K},\mathcal{C})\longrightarrow M_{RRB}(\mathcal{A})$$ are isomorphisms. Thus, for each $i$, there is a homomorphism $(\psi_i,\eta_i) \in \Hom_{RRB}(\mathcal{K},\mathcal{C})$ such that $(\psi_i\tau^1_1,\eta_i\tau^1_2,\psi_i\rho^1,\eta_i\chi^1)$ is cohomologous to $(\tau^2_{1_i},\tau^2_{2_i},\rho^2_i,\chi^2_i)$. Thus, there are maps $\theta_{1_i}:A\longrightarrow\mathbb{C}^{\times}$ and $\theta_{2_i}:B\longrightarrow\mathbb{C}^{\times}$ such that 
\begin{equation}\label{intermediate equation}
 \big( (\psi_i\tau^1_1) \cdot \partial (\theta_{1_i}),\, (\eta_i\tau^1_2) \cdot \partial (\theta_{2_i}), \,(\psi_i\rho^1) \cdot \lambda_1,\, (\eta_i\chi^1) \cdot \lambda_2 \big)=(\tau^2_{1_i},\tau^2_{2_i},\rho^2_i,\chi^2_i),
 \end{equation}
	where $\cdot$ is the point-wise product of maps, and $\lambda_1$ and $\lambda_2$ are defined in \eqref{t3} and \eqref{t4}, respectively. We claim that $\Z^{\phi_1}_{R_1}(H_1)=\Z^{\phi_2}_{R_2}(H_2)$ as sets. Let $(a,k) \in \Z^{\phi_1}_{R_1}(H_1)$. By Lemma \ref{Centrecondn}, we have $a\in \Z(A)$ and $\tau^1_1(a,x)=\tau^1_1(x,a)$ for all  $x\in A$. Then, for all $x\in A$, we have
	\begin{eqnarray*}
		\tau^2_{1_i}(a,x)&=&\psi_i(\tau^1_1(a,x)) \, \theta_{1_i}(x)\theta_{1_i}(ax)^{-1}\theta_{1_i}(a)\\
		&=&\psi_i(\tau^1_1(x,a)) \,\theta_{1_i}(a)\theta_{1_i}(xa)^{-1}\theta_{1_i}(x)\\
		&=&\tau^2_{1_i}(x,a).
	\end{eqnarray*}
Hence, $\tau^2_1(a,x)=\tau^2_1(x,a)$ for all $x\in A$. Again, by Lemma \ref{Centrecondn}, we have $\beta_{T(a)}(x)=x$ and $\rho^1(x,T(a))=1$ for all $x\in A$. This gives
$$ \rho^2_i(x,T(a))=\psi_i(\rho^1(x,T(a)))\theta_{1_i}(x)\theta_{1_i}(\beta_{T(a)}(x))^{-1}=1$$
for all $x \in A$, and hence $\rho^2(x,T(a))=1$ for all $x\in A$.	Similarly, it can be proved that $\rho^2(a,b)=1$ for all $b\in B$, and hence  $(a,k) \in \Z^{\phi_2}_{R_2}(H_2)$. By interchanging the roles, we obtain  $\Z^{\phi_1}_{R_1}(H_1)=\Z^{\phi_2}_{R_2}(H_2)$.
\par

Next, we claim that $\Ker(\phi_1)=\Ker(\phi_2)$ as sets. If $(b,l)\in\Ker(\phi_1)$, then $\phi_{1_{(b,l)}}(a,k)=(a,k)$ for all $(a,k)\in A\times_{\tau^1_1}K$. This gives $\beta_b(a)=a$  and $\rho^1(a,b)=1$ for all $a\in A$. Since $\psi_i\rho^1\lambda_1=\rho^2_i$, we get  $\rho^2_i(a,b)=1$ for all $i$. This gives $\rho^2(a,b)=1$ for all $a\in A$, and hence $(b,l)\in\Ker(\phi_2)$. Interchanging the roles, we obtain $\Ker(\phi_1)=\Ker(\phi_2)$.
\par

The preceding two claims yield well-defined group homomorphisms 
	$$\tilde{\psi}_1:H_1/\Z^{\phi_1}_{R_1}(H_1)\longrightarrow H_2/\Z^{\phi_2}_{R_2}(H_2) \quad\textrm{and} \quad \tilde{\eta}_1:G_1/\Ker(\phi_1)\longrightarrow G_2/\Ker(\phi_2)$$	
	 given by
$$	\tilde{\psi_1} \big(\widetilde{(a,k)} \big)=\overline{(a,k)} \quad\textrm{and} \quad \tilde{\eta}_1 \big(\widetilde{(b,l)}\big)=\overline{(b,l)}.$$
We claim that the tuple
	$$(\tilde{\psi}_1, \tilde{\eta}_1): (H_1, G_1, \phi_1, R_1) / \Z(H_1, G_1,  \phi_1, R_1 ) \rightarrow (H_2, G_2, \phi_2, R_2) / \Z(H_2, G_2, \phi_2, R_2)$$
is, in fact, an isomorphism of relative Rota-Baxter groups. If $\widetilde{(a,k)} \in H_1/\Z^{\phi_1}_{R_1}(H_1)$, then 
$$ \tilde{\eta}_1\overline{R}_1\big(\widetilde{(a,k)} \big) =\tilde{\eta}_1 \big(\widetilde{R_1(a,k)} \big)=\tilde{\eta}_1\widetilde{\big(T(a),\chi^1(a)S(k) \big)}=\overline{\big(T(a),\chi^1(a)S(k) \big)}$$
and
$$ \overline{R}_2\tilde{\psi}_1\big(\widetilde{(a,k)}\big)= \overline{R}_2 \big(\overline{(a,k)}\big)=\overline{(T(a),\chi^2(a)S(k))}.$$
Since we can write $(T(a),\chi^1(a)S(k))=(T(a),\chi^2(a)S(k))(1,\chi^1(a)\chi^2(a)^{-1})$, where $(1,\chi^1(a)^{-1}\chi^2(a)) \in \Ker(\phi_2)$, it follows that $\overline{(T(a),\chi^1(a)S(k))}=\overline{(T(a),\chi^2(a)S(k))}$.  Next, for $\widetilde{(b,l)} \in G_1/\Ker(\phi_1)$, we have
$$ \tilde{\psi}_1\overline{\phi}_{1_{\widetilde{(b,l)}}}\widetilde{(a,k)} = \tilde{\psi}_1\widetilde{\big(\beta_b(a),\rho^1(a,b)k \big)}=\overline{(\beta_b(a),\rho^1(a,b)k)}$$
and
$$ \overline{\phi}_{2_{\tilde{\eta}_{1}\widetilde{(b,l)}}}\tilde{\psi_1}\widetilde{(a,k)} = \overline{\phi}_{2_{\overline{(b,l)}}}\overline{(a,k)}=\overline{(\beta_b(a),\rho^2(a,b)k)}.$$
Since $(\beta_b(a),\rho^1(a,b)k)=(\beta_b(a),\rho^2(a,b)k)(1,\rho^1(a,b)\rho^2(a,b)^{-1})$, where $(1,\rho^1(a,b)\rho^2(a,b)^{-1}) \in \Z^{\phi_2}_{R_2}(H_2)$, it follows that   $\overline{(\beta_b(a),\rho^1(a,b)k)}=\overline{(\beta_b(a),\rho^2(a,b)k)}$.   Thus, $(\tilde{\psi}_1, \tilde{\eta}_1)$ is a homomorphism of relative Rota-Baxter groups.
 Clearly, both  $\tilde{\psi_1}$ and $\tilde{\eta_1}$ are bijective, and hence  $(\tilde{\psi}_1, \tilde{\eta}_1)$ is an isomorphism of relative Rota-Baxter groups.
\par 	

Consider the relative Rota-Baxter  group $\tilde{\mathcal{H}}_2= \big(A\times_{\tau^2_1} (\mathbb{C}^{\times})^{m}, \, B\times_{\tau^2_2}(\mathbb{C}^{\times})^{m}, \, \phi_2, \, R_2 \big)$, which is a central extension of $\mathcal{A}$ by $\big((\mathbb{C}^{\times})^m,\, (\mathbb{C}^{\times})^m,\, \alpha^0\times\dots\times\alpha^0,\, S^0\times\dots\times S^0 \big)$ corresponding to the 2-cocycle $(\tau^2_1,\tau^2_2,\rho^2,\chi^2)$. Here we abuse the notation and the map $\phi_2:B\times_{\tau^2_2}(\mathbb{C}^{\times})^{m}\to \Aut(A\times_{\tau^2_1} (\mathbb{C}^{\times})^{m})$ is given by $\phi_{2_{(b,l)}}(a,k)=(\beta_b(a),\rho^2(a,b)\,k)$, whereas the map $R_2:A\times_{\tau^2_1} (\mathbb{C}^{\times})^{m}\to B\times_{\tau^2_2}(\mathbb{C}^{\times})^{m}$ is given by $R_2(a,k)=(T(a),\chi^2(a) \,k)$.  Consider the maps
\begin{eqnarray*}
\psi\colon K\to (\C^{\times})^m & \textrm{given by}& \psi(k) = \big(\psi_1(k),\dots , \psi_m(k) \big), \\
\eta\colon L\to (\C^{\times})^m & \textrm{given by}& \eta(l) = \big(\eta_1(l),\dots , \eta_m(l) \big), \\
\theta_1\colon A\to (\C^{\times})^m & \textrm{given by}& \theta_1(a) = \big(\theta_{1_1}(a),\dots , \theta_{1_m}(a) \big), \\
\theta_2\colon B\to (\C^{\times})^m & \textrm{given by}& \theta_2(b) = \big(\theta_{2_1}(b),\dots , \theta_{2_m}(b) \big).
\end{eqnarray*}
This gives
$$(\tilde{\psi}_2, \tilde{\eta}_2): \big(A\times_{\tau^1_1} K, \, B\times_{\tau^1_2} L, \, \phi_1, \, R_1 \big)\rightarrow \big(A\times_{\tau^2_1}(\C^{\times})^m, \, B\times_{\tau^2_2} (\C^{\times})^m, \, \phi_2, \, R_2 \big)$$ defined by
$$\tilde{\psi}_2(a,k)=\big(a,\psi(k)\theta_1(a)^{-1} \big) \quad \textrm{and} \quad \tilde{\eta}_2(b,l)= \big(b,\eta(l)\theta_2(b)^{-1} \big).$$
For $a_1, a_2 \in A$ and $k_1, k_2 \in K$, we have
\begin{eqnarray*}
\tilde{\psi}_2 \big((a_1,k_1)(a_2,k_2) \big)&=&\tilde{\psi}_2 \big(a_1a_2,k_1k_2\tau^1_1(a_1,a_2) \big)\\
&=& \big(a_1a_2,\psi(k_1k_2\tau^1_1(a_1,a_2))\theta_1(a_1,a_2)^{-1} \big)\\
&=& \big(a_1a_2,\psi(k_1k_2)\psi(\tau^1_1(a_1,a_2))\theta_1(a_1a_2)^{-1} \big)\\
&=& \big(a_1a_2,\psi(k_1k_2)\psi(\tau^1_1(a_1,a_2))\theta_1(a_2)\theta_1(a_1a_2)^{-1}\theta_1(a_1)\theta_1(a_1)^{-1}\theta_{1}(a_2)^{-1} \big)\\
&=& \big(a_1a_2,\psi(k_1k_2)\tau^2_1(a_1,a_2)\theta_1(a_1)^{-1}\theta_{1}(a_2)^{-1} \big),~\textrm{using}~ \eqref{intermediate equation}\\
&=& \big(a_1a_2,\psi(k_1)\theta_1(a_1)^{-1}\psi(k_2)\theta_1(a_2)^{-1}\tau^2_1(a_1,a_2) \big)\\
&=& \big(a_1,\psi(k_1)\theta_1(a_1)^{-1} \big) \big(a_2,\psi(k_2)\theta_1(a_2)^{-1} \big)\\
&=& \tilde{\psi}_2(a_1,k_1) \, \tilde{\psi_2}(a_2,k_2),
\end{eqnarray*}
which proves that $\tilde{\psi}_2$ is a group homomorphism. Similarly, we can prove that $\tilde{\eta}_2$ is a group homomorphism. We claim that $\tilde{\eta}_2R_1=R_2\tilde{\psi}_2$ and $\tilde{\psi}_2\phi_{1_{(b,l)}}=\phi_{2_{\tilde{\eta}_2(b,l)}}\tilde{\psi}_2$ for all $(b, l) \in  B\times_{\tau^1_2} L$. For $(a,k) \in A\times_{\tau^1_1} K$, we see that
\begin{eqnarray*}
\tilde{\eta}_2R_1(a,k)&=&\tilde{\eta}_2 \big(T(a),\chi^1(a)S(k) \big)\\
&=& \big(T(a),\eta(\chi^1(a)S(k))\theta_2(T(a))^{-1} \big)\\
&=& \big(T(a),\eta(\chi^1(a))\eta(S(k))\theta_2(T(a))^{-1} \big).
\end{eqnarray*}	 
Observe that 
$$ \eta(\chi^1(a))\eta(S(k))\theta_2(T(a))^{-1}= \big(\eta_1(\chi^1(a))\eta_1(S(k))\theta_{2_1}(T(a))^{-1}, \dots,\eta_m(\chi^1(a))\eta_m(S(k))\theta_{2_m}(T(a))^{-1} \big).$$
Further, by \eqref{intermediate equation}, we obtain
\begin{eqnarray*}
\eta_i(\chi^1(a))\eta_i(S(k))\theta_{2_i}(T(a))^{-1}&=&\eta_i(\chi^1(a))S^0(\theta_{1_i}(a))\theta_{2_i}(T(a))^{-1}S^0(\theta_{1_i}(a))^{-1}\eta_i(S(k))\\
&=&\chi^2_i(a)S^0(\theta_{1_i}(a))^{-1}\eta_i(S(k))
\end{eqnarray*}
for each $i$. Thus, we can write
$$
\tilde{\eta}_2R_1(a,k)=\big(T(a),\eta(\chi^1(a))\eta(S(k))\theta_2(T(a))^{-1} \big) = \big(T(a),\chi^2(a)(S^0\times\dots \times S^0)(\theta_{1}(a)^{-1})\eta(S(k)) \big).$$
On the other hand, we have
\begin{eqnarray*}
R_2\tilde{\psi}_2(a,k)&=&R_2 \big(a,\psi(k)\theta_1(a)^{-1} \big)\\
&=& \big(T(a),\chi^2(a)(S^0\times\dots \times S^0)(\psi(k)\theta_1(a)^{-1}) \big)\\
&=& \big(T(a),\chi^2(a)S(\psi(k))(S^0\times\dots \times S^0)(\theta_1(a)^{-1})\big).
\end{eqnarray*}
Since $(\psi_i,\eta_i)\in\Hom_{RRB}(\mathcal{K},\mathcal{C})$, we obtain
$$ \eta(S(k)) = \big(\eta_1(S(k)),\dots,\eta_m(S(k))\big)= \big(S^0(\psi_1(k)),\dots,S^0(\psi_m(k)) \big) =S(\psi(k)),$$
which implies that $\tilde{\eta}_2R_1=R_2\tilde{\psi}_2$. Next, we see that
\begin{eqnarray*}
\tilde{\psi}_2\phi_{1_{(b,l)}}(a,k)&=&\tilde{\psi}_2 \big(\beta_b(a),\rho^1(a,b)k \big)\\
&=& \big(\beta_b(a),\psi(\rho^1(a,b)k)\theta_1(\beta_b(a))^{-1} \big)\\
&=& \big(\beta_b(a),\psi(\rho^1(a,b))\psi(k)\theta_1(\beta_b(a))^{-1}\theta_1(a)\theta_1(a)^{-1} \big)\\
&=& \big(\beta_b(a),\rho^2(a,b)\psi(k)\theta_1(a)^{-1} \big), ~\textrm{by}~ \eqref{intermediate equation}\\
&=&	\phi_{2_{\tilde{\eta}_2(b,l)}}(a,\psi(k)\theta_1(a)^{-1})\\
&=& 	\phi_{2_{\tilde{\eta}_2(b,l)}}\tilde{\psi}_2(a,k).
\end{eqnarray*}
Thus, $(\tilde{\psi}_2, \tilde{\eta}_2)$ is a homomorphism of relative Rota-Baxter groups. In view of Remark \ref{Generators}, we see that
	\begin{eqnarray*}
		\tilde{\psi_2} \big((a_1a_2a_1^{-1}a_2^{-1}, ~\tau^1_1(a_1,a_2)\tau^1_1(a_1,a_1^{-1})^{-1}\tau^1_1(a_1a_2,a_1^{-1})\tau^1_1(a_2,a_2^{-1})^{-1}\tau^1_1(a_1a_2a_1^{-1},a_2^{-1}))\big)\\
		=\big(a_1a_2a_1^{-1}a_2^{-1},~\tau^2_1(a_1,a_2)\tau^2_1(a_1,a_1^{-1})^{-1}\tau^2_1(a_1a_2,a_1^{-1})\tau^2_1(a_2,a_2^{-1})^{-1}\tau^2_1(a_1a_2a_1^{-1},a_2^{-1})\big),
	\end{eqnarray*}
and
$$ \tilde{\psi_2} \big((\beta_b(a)a^{-1},~\rho^1(a,b)\tau^1_1(a,a^{-1})^{-1}\tau^1_1(\beta_b(a),a^{-1})) \big)= \big(\beta_b(a)a^{-1}, ~\rho^2(a,b)\tau^2_1(a,a^{-1})^{-1}\tau^2_1(\beta_b(a),a^{-1}) \big).$$
Thus, $\tilde{\psi_2}$ maps the generators of $H_1^{\phi_1}$ to the generators of $H_2^{\phi_2}$, and hence $(\tilde{\psi}_2, \tilde{\eta}_2)$ descends to the homomorphism $$(\tilde{\psi}_2|, \tilde{\eta}_2|): \big(H_1^{\phi_1},R_1(H_1^{\phi_1}), \phi_1, R_1 \big)\rightarrow \big(H_2^{\phi_2} ,R_2(H_2^{\phi_2}), \phi_2, R_2\big)$$ of relative Rota-Baxter groups.  In other words, $(\tilde{\psi}_2|, \tilde{\eta}_2|): \I((H_1,G_1,\phi_1,R_1)') \to \I((H_2,G_2,\phi_2,R_2)')$ is a homomorphism of relative Rota-Baxter groups. It turns out that $\tilde{\psi}_2|: H_1^{\phi_1} \to H_2^{\phi_2}$is an isomorphism of groups. To see this, we construct group  homomorphism  $\tilde{\psi}_3:A\times_{\tau^2_1} K\rightarrow A\times_{\tau^1_1}(\C^{\times})^m$ as before. Then $\tilde{\psi}_3|: H_2^{\phi_2} \to H_1^{\phi_1}$is the desired inverse of $\tilde{\psi}_2|$. The commutativity of the diagram \ref{cd1} follows from the above assertions. By interchanging roles of $\mathcal{H}_1$ and  $\mathcal{H}_2$,  we can see that $\mathcal{H}_2$ is also weakly isoclinic to  $\mathcal{H}_1$. This completes the proof.
\end{proof}

 \begin{defn}
	Let $\mathcal{A}=(A,B,\beta,T)$ be a relative Rota-Baxter group. A Schur cover of $\mathcal{A}$ is a relative Rota-Baxter group $\mathcal{H}=(H,G,\phi,R)$ such that there exists a central extension
	$$ \quad  {\bf 1} \longrightarrow \mathcal{K} \stackrel{(i_1, i_2)}{\longrightarrow}  \mathcal{H} \stackrel{(\pi_1, \pi_2)}{\longrightarrow} \mathcal{A} \longrightarrow {\bf 1}$$
	such that $\mathcal{K} \subseteq \mathcal{H}^{'}$, where $\mathcal{K}= \big(M_{RRB}(\mathcal{A}), \,M_{RRB}(\mathcal{A}), \,\alpha^0|\times\dots \times \alpha^0|, \,S^0|\times\dots \times S^0| \big)$.
\end{defn}

\begin{remark}
When $\mathcal{A}$ is a bijective relative Rota-Baxter group, then the preceding definition aligns seamlessly with the concept of Schur covers of skew left braces introduced by Letourmy and Vendramin
in \cite{TV23}. This alignment is achieved by recognizing every skew left brace as a bijective relative Rota-Baxter group under the isomorphism between the two categories as established in \cite[Theorem 4.6]{NM1}.
\end{remark}

\begin{prop}\label{SC2}
Let   $\mathcal{H}$ be a Schur cover of a finite relative Rota-Baxter group $\mathcal{A}$ and
$$
  {\bf 1} \longrightarrow \mathcal{K}=(K,L, \alpha,S ) \stackrel{(i_1, i_2)}{\longrightarrow}  \mathcal{H}=(H,G, \phi, R) \stackrel{(\pi_1, \pi_2)}{\longrightarrow} \mathcal{A}=(A,B, \beta, T) \longrightarrow {\bf 1}
$$
the corresponding central extension. Then the transgression map $ \Tra : {\Hom_{RRB}{(\mathcal{K},\mathcal{C})}}{\longrightarrow}{\Ho^2_{RRB}{(\mathcal{A}, \mathcal{C})}}$ is an isomorphism.
\end{prop}

\begin{proof}
Suppose that $\mathcal{H}$ is a Schur cover of $\mathcal{A}$. Since $\mathcal{K}\subseteq \mathcal{H}'$, by Lemma \ref{SC1}, every homomorphism $\mathcal{H}\to\mathcal{C}$ restricts to the trivial homomorphism $\mathcal{K}\to\mathcal{C}$. It follows from Theorem \ref{HSS} that the map $ \Tra : {\Hom_{RRB}{(\mathcal{K},\mathcal{C})}}{\longrightarrow}{\Ho^2_{RRB}{(\mathcal{A}, \mathcal{C})}}$ is injective. Since $\mathcal{A}$ is finite, it follows from Theorem \ref{VI2} that $K=M_{RRB}(\mathcal{A})$ is finite, and hence $\Hom_{Group}(K, \mathbb{C}^{\times})\cong K$. Further, as both $\mathcal{K}$ and $\mathcal{C}$ are bijective relative Rota-Baxter groups, using \cite[Proposition 4.5]{NM1}, we obtain
$$\Hom_{RRB}(\mathcal{K},\mathcal{C}) \cong  \Hom_{Group}(K, \mathbb{C}^{\times}) \cong K=   M_{RRB}(\mathcal{A}).$$ 
Hence, it follows that the map $\Tra$ is an isomorphism. 
\end{proof}

\begin{prop}\label{Transgression Isomorphism and Schur cover}
Let $\mathcal{A}$ be a  finite bijective relative Rota-Baxter group and
$$
{\bf 1} \longrightarrow \mathcal{K}=(K,L, \alpha,S ) \stackrel{(i_1, i_2)}{\longrightarrow}  \mathcal{H}=(H,G, \phi, R) \stackrel{(\pi_1, \pi_2)}{\longrightarrow} \mathcal{A}=(A,B, \beta, T) \longrightarrow {\bf 1}
$$ a central extension of $\mathcal{A}$ by $\mathcal{K}=\big(M_{RRB}(\mathcal{A}), \,M_{RRB}(\mathcal{A}), \,\alpha^0|\times\dots \times \alpha^0|, \,S^0|\times\dots \times S^0| \big)$. Then $\mathcal{H}$ is a Schur cover of $\mathcal{A}$ if and only if the transgression map  $ \Tra : {\Hom_{RRB}{(\mathcal{K},\mathcal{C})}}{\longrightarrow}{\Ho^2_{RRB}{(\mathcal{A}, \mathcal{C})}}$ is an isomorphism.
\end{prop}
\begin{proof}
The forward implication is simply  Proposition \ref{SC2}.	Conversely, if $\Tra$ is an isomorphism, then it follows from Theorem \ref{HSS} that the map $\Res : \Hom(\mathcal{H},\mathcal{C})\to \Hom(\mathcal{K},\mathcal{C})$ is trivial. Thus, by Lemma \ref{converse SC1}, we have $\mathcal{K}\subseteq\mathcal{H}'$, and  hence $\mathcal{H}$ is a Schur cover of $\mathcal{A}$.
\end{proof}

Theorem \ref{existence of Transgression Isomorphism} and Proposition \ref{Transgression Isomorphism and Schur cover} yield the following result.

\begin{cor}
Every finite bijective relative Rota Baxter group admits at least one Schur cover. 
\end{cor}

Theorem \ref{thm:almost isoclinism} and Proposition \ref{SC2} yield the following result.

\begin{cor}
Any two Schur covers of a finite bijective relative Rota-Baxter group are weakly isoclinic.
\end{cor}

\begin{remark}
In view of Theorem \ref{almost isoclinism rrbg implies isoclinism slb}, the Schur covers of the induced skew left brace are isoclinic. This aligns with the result established in \cite[Theorem 3.19]{TV23}.
	\end{remark}

\begin{ack}
	{\rm Pragya Belwal thanks UGC for the PhD research fellowship and Nishant Rathee thanks IISER Mohali for the institute post doctoral fellowship. Mahender Singh is supported by the SwarnaJayanti Fellowship grants DST/SJF/MSA-02/2018-19 and SB/SJF/2019-20/04.}
\end{ack}

\section{Declaration}
The authors declare that there is no data associated with this paper and that there are no conflicts of interest.

\end{document}